\documentclass[11pt,a4paper]{article}

\usepackage[T1]{fontenc}
\usepackage{microtype}
\usepackage{newtxtext}
\usepackage[scaled=0.92]{helvet}
\usepackage{dsfont}

\usepackage{amsmath}
\usepackage{mathtools}
\usepackage{amsthm}
\usepackage{newtxmath}
\usepackage{bm}
\usepackage{xfrac}

\usepackage{geometry}
\usepackage{graphicx}
\usepackage{caption}
\usepackage[section]{placeins}
\usepackage{enumitem}
\usepackage{xcolor}
\usepackage{tikz}
\usetikzlibrary{patterns,calc,backgrounds}

\usepackage{natbib}
\usepackage[nottoc]{tocbibind}
\usepackage[hidelinks]{hyperref}
\usepackage[capitalize]{cleveref}

\title{Unrestricted conditional maximum likelihood estimation for a Fréchet regression model}
\author{%
  Marcio Reverbel\thanks{Corresponding author. Email: \texttt{marcio.reverbel@kuleuven.be}}
  \\[0.3ex]
  \small Department of Mathematics, KU Leuven \\
  \small Celestijnenlaan 200B, 3001 Heverlee, Belgium
  \and
  Johan Segers\thanks{Email: \texttt{jjjsegers@kuleuven.be}}
  \\[0.3ex]
  \small Department of Mathematics, KU Leuven \\
  \small Celestijnenlaan 200B, 3001 Heverlee, Belgium \\
  \small LIDAM/ISBA, UCLouvain, Louvain-la-Neuve, Belgium
}
\date{\today}

\theoremstyle{plain}
\newtheorem{condition}{Condition}[section]
\newtheorem{lemma}{Lemma}[section]
\newtheorem{definition}{Definition}[section]
\newtheorem{proposition}{Proposition}[section]
\newtheorem{theorem}{Theorem}[section]

\crefname{condition}{Condition}{Conditions}
\Crefname{condition}{Condition}{Conditions}

\crefname{equation}{Eq.}{Eqs.}
\Crefname{equation}{Eq.}{Eqs.}

\newenvironment{lproof}[1][Proof]{%
  \begin{proof}[#1]%
}{%
  \end{proof}%
}
\newenvironment{pproof}[1][Proof]{%
  \begin{proof}[#1]%
}{%
  \end{proof}%
}

\DeclareMathOperator*{\argmax}{arg\,max}
\DeclareMathOperator{\hzn}{hzn}
\DeclareMathOperator{\csm}{csm}
\DeclareMathOperator{\dir}{dir}

\newcommand{\pr}{\operatorname{\mathsf{P}}}
\newcommand{\ex}{\operatorname{\mathsf{E}}}
\newcommand{\var}{\operatorname{\mathsf{Var}}}
\newcommand{\norm}[1]{\left\| #1 \right\|}

\newcommand{\mm}[2]{m_{#1}({#2})}
\newcommand{\m}[1]{m_{\theta}({#1})}
\newcommand{\tr}{\intercal}
\newcommand{\tm}[1][l]{\tilde m^{(#1)}}
\newcommand{\om}[1][l]{\overline m^{(#1)}}

\newcommand{\eps}{\varepsilon}
\newcommand{\Eta}{\bm{\mathrm H}}
\newcommand{\sTheta}{\Theta^\circ}
\newcommand{\mTheta}{\Theta_0}
\newcommand{\bmeta}{{\bm{\eta}}}
\newcommand{\bbeta}{{\bm{\beta}}}
\newcommand{\bgamma}{{\bm{\gamma}}}
\newcommand{\Ud}{{U_{\bmeta_\dagger}}}
\newcommand{\Uhzn}{{U_{\hzn}}}
\newcommand{\mle}{\hat \bmeta_n}
\newcommand{\est}{\hat \theta}
\newcommand{\nme}{\tilde \bmeta_n}
\newcommand{\Fi}{\cI_{\bmeta_\dagger}}
\newcommand{\Hes}{H_{\bmeta_\dagger}}
\newcommand{\Pl}{P_Z^{\otimes l}}

\newcommand{\reals}{\mathbb{R}}
\newcommand{\naturals}{\mathbb{N}}
\newcommand{\Rp}{{\reals}^{p}}
\newcommand{\Rpos}{\reals_{>0}}
\newcommand{\sphere}[1][p-1]{\mathbb{S}^{#1}}
\newcommand{\cA}{\mathcal A}
\newcommand{\cF}{\mathcal F}
\newcommand{\cI}{\mathcal I}
\newcommand{\cM}{\mathcal M}
\newcommand{\cN}{\mathcal N}
\newcommand{\cW}{\mathcal W}
\newcommand{\cZ}{\mathcal Z}
\newcommand{\leb}{\operatorname{Leb}}

\newcommand{\0}{\bm 0}
\newcommand{\ba}{\bm a}
\newcommand{\bb}{\bm b}
\newcommand{\bt}{\bm t}
\newcommand{\bv}{\bm v}
\newcommand{\bw}{\bm w}
\newcommand{\bx}{\bm x}
\newcommand{\by}{\bm y}
\newcommand{\bz}{\bm z}
\newcommand{\beps}{\bm \eps}
\newcommand{\bG}{\bm G}
\newcommand{\bW}{\bm W}
\newcommand{\bX}{\bm X}
\newcommand{\bY}{\bm Y}
\newcommand{\bZ}{\bm Z}
\newcommand{\y}{\bm{\ln y}}
\newcommand{\Y}{\bm{\ln Y}}

\newcommand{\mA}{\mathbf A}
\newcommand{\mO}{\mathbf O}
\newcommand{\mQ}{\mathbf Q}
\newcommand{\mX}{\mathbf X}
\newcommand{\oX}{\mathbf x}
\newcommand{\mLambda}{\bm\Lambda}

\newcommand{\ad}{\alpha_\dagger}
\newcommand{\bd}{\bbeta_\dagger}

\newcommand{\abr}[1]{\left|#1\right|}
\newcommand{\cbr}[1]{\left\{#1\right\}}
\newcommand{\rbr}[1]{\left(#1\right)}
\newcommand{\sbr}[1]{\left[#1\right]}

\newcommand{\mtext}[1]{\quad \text{#1} \quad}

\newcommand{\rank}{\operatorname{rank}}

\newcommand{\Frechet}{\operatorname{Fr\acute{e}chet}}
\newcommand{\Gumbel}{\operatorname{Gumbel}}
\newcommand{\Exp}{\operatorname{Exp}}
\newcommand{\Normal}{\operatorname{N}}
\newcommand{\pto}{\xrightarrow{\pr}}
\newcommand{\dto}{\rightsquigarrow}

\begin{document}

\maketitle

\begin{abstract}
  Heavy-tailed response variables are routinely modeled as a function of
  covariates,
  but the asymptotic
  theory available for the resulting maximum likelihood estimators sidesteps
  certain technical difficulties by restricting the parameters to a compact set,
  which is not the natural parameter space. We consider a Fr\'echet model whose
  scale parameter depends log-linearly on a vector of covariates, while the tail
  parameter is constant, and we study the conditional maximum likelihood estimator
  of the tail parameter and the regression coefficients jointly. We show that this
  estimator is consistent over the full, noncompact parameter space, and that it is
  asymptotically normal, with the inverse Fisher information as limiting covariance
  matrix. Consistency is obtained by a specific compactification of the parameter space relying on the cosmic closure of Euclidean space. The extension of Wald's consistency theorem needed for that step, from single observations to blocks of observations, is stated in an
  abstract setting and is possibly of independent interest.
\end{abstract}


\section{Introduction}
\label{sec:introduction}

Statistical modeling of heavy-tailed phenomena plays a central role in finance, environmental sciences, and insurance mathematics. Extreme Value Theory (EVT) provides a mathematically grounded framework for describing tail behavior and extrapolating beyond the observed range. Two main approaches are commonly used to model extremes: the \emph{block maxima method} (BM), which considers the maxima of equal-sized, nonoverlapping blocks of observations, and the \emph{peaks-over-threshold method} (POT), which considers observations exceeding a sufficiently high threshold. Under appropriate conditions, the limiting distribution of suitably normalized block maxima belongs to the three-parameter Generalized Extreme Value (GEV) family, while excesses over a sufficiently high threshold are asymptotically described by a Generalized Pareto distribution (GP). These distributions are therefore commonly used as approximating models for extreme observations. The heavy-tailed case corresponds to a positive extreme value index $\xi > 0$, where $\xi$ is the shape parameter common to both the GEV and GP families. In the GEV case, this can be reparameterized as a Fréchet distribution through an appropriate location-scale transformation.

When it comes to EVT, even classical results such as consistency and asymptotic normality of the maximum likelihood estimator (MLE) have proved difficult to obtain. The asymptotic normality of the MLE of the GEV distribution, and the consistency and asymptotic normality of the MLE for block maxima extracted from a time series in the max domain of attraction of the Fréchet distribution, were established only in the last decade \citep{BucherSegers2017GEV, BucherSegers2018}. Existence and uniqueness of the MLE for the GP distribution fitted to excesses over high thresholds were established only in 2025 \citep{DombryPadoanRizzelli2025}.
Despite substantial progress, several fundamental challenges remain open. Applications increasingly require regression-type extensions with covariate-dependent model parameters, such as the modeling of climate extremes based on geographical and meteorological variables \citep[see, for instance,][]{coles2001introduction,cooley2009extreme,davison1990models,smith1989extreme,zanger2024regional}. Popular model specifications include a linear model
for the location parameter and log-linear models
for the scale parameters of the GEV and GP distributions.
Covariate models for the shape parameter are less commonly used. Such regression models have been extended in several directions: random rather than fixed effects (i.e., the regression coefficients are themselves realizations of latent random variables), and generalized additive models, which provide more flexibility in modeling the link between regressor and response \citep{chavez2005generalized}. In fixed-design distributional regression, consistency of M-estimators was established by \citet{BucherSegersStaud2025}, extending previous results to the challenging EVT setting. This includes strong consistency of the \emph{Conditional Maximum Likelihood Estimator} (CMLE) for GEV and GP models, obtained under the assumption that the true parameters lie in a compact subset of the parameter space.

The purpose of this paper is to establish the consistency and asymptotic normality of the CMLE over the unrestricted parameter space for a Fréchet model with covariate-dependent scale parameter. We adopt a log-linear model for the scale parameter, as is commonly done in practice. Let $\{(Y_i, \bX_i) : i = 1, \ldots, n\}$ be independent and identically distributed random pairs taking values in $\Rpos \times \Rp$ and such that for all $\bx \in \Rp$ the conditional distribution of $Y_i$ given $\bX_i = \bx$ is Fréchet with shape parameter $\alpha \in \Rpos$ and scale parameter $\exp(\bbeta^\tr \bx)$, that is,
\begin{equation}
    \label{eq:model}
    \pr \rbr{ Y_i \le y \mid \bX_i = \bx}
    =
    \exp \sbr{ - \rbr{ \frac{y}{\exp(\bbeta^\tr \bx)}}^{-\alpha}},
    \qquad y \in \Rpos,
\end{equation}
where $\bbeta \in \Rp$ is the vector of regression coefficients. The full parameter vector is
\[
    \bmeta
    = (\alpha, \bbeta) \in \Eta := \Rpos \times \Rp.
\]
In this model, the difficulties associated with the parameter-dependent support of the general GEV family are absent. As in \citet{BucherSegers2017GEV}, we work under exact specification and therefore do not address the approximation error inherent in the domain-of-attraction framework. The difficulty considered is of a different nature: the natural parameter space for the regression coefficients and tail parameter is noncompact. Consequently, standard consistency results for M-estimators on compact parameter spaces do not apply.

Our main result establishes the consistency of the CMLE under fairly broad assumptions on the covariates. The main device is a compactification of the parameter space, which allows us to show that sequences escaping to the boundary cannot asymptotically maximize the likelihood. We also provide an extension of Wald's consistency theorem (Theorem~5.14 in \citealp{van2000asymptotic}), which we use to establish consistency of the CMLE. This extension may be of independent interest. Once consistency is obtained, asymptotic normality of the CMLE follows from standard theory (Theorem~5.41 in the same reference).

The remainder of the paper is organized as follows. \Cref{sec:framework} introduces the model, the associated objective function, the identifiability condition on the covariates, and the existence, uniqueness and definition of the CMLE. \Cref{sec:wald} presents the extension of Wald's consistency theorem in an abstract setting. \Cref{sec:consistency} constructs the compactification of the parameter space and establishes consistency of the CMLE, and \Cref{sec:asymptotic_normality} establishes its asymptotic normality.
\Cref{sec:conclusion} concludes. All proofs are collected in \Cref{sec:proofs_framework,sec:proofs_wald,sec:proofs_consistency,sec:proofs_normality}.

Throughout we use the following notation. We write $\naturals = \cbr{1, 2, \ldots}$ for the set of positive integers and $P_W$ for the distribution of a random object $W$. The $p$-variate observations are seen as \emph{column} vectors, so that the design matrix $\mX$ formed by a sample or a block of them has the transposed observations $\bX_i^\tr$ as its rows; this is in line with our use of $\ex[\bX \bX^\tr]$ later on. Boldface capitals denote random elements and the matching lowercase letters their realizations, as in $\mX$ and $\oX$, or $\bX$ and $\bx$.

\section{Fréchet Regression Model and Conditional Maximum Likelihood Estimator}
    \label{sec:framework}
Consider the model presented in \Cref{eq:model}. We are interested in showing consistency and asymptotic normality of the CMLE, $\mle$, for $\bmeta = (\alpha, \bbeta) \in \Eta.$ This section collects what the model itself provides: the objective function and its properties (\Cref{subsec:objective_function}), the condition on the covariates under which the parameter is identifiable (\Cref{subsec:identifiability}), the existence, uniqueness and definition of the estimator at a fixed sample size (\Cref{subsec:cmle}), and the asymptotic criterion function whose unique point of maximum is the true parameter (\Cref{subsec:asymptotic_criterion}).

\subsection{The Model and the Objective Function}
    \label{subsec:objective_function}
The (conditional) density function associated with \cref{eq:model} is
\begin{equation}
    \label{eq:pdf}
    f_{\bmeta}^{Y\mid \bX}(y \mid \bx)
    = \frac{\alpha}{\sigma(\bx)}
    \left(\frac{y}{\sigma(\bx)}\right)^{-1-\alpha}
    \exp\!\left[-\left(\frac{y}{\sigma(\bx)}\right)^{-\alpha}\right],
    \qquad y>0
\end{equation}
where $\sigma(\bx) = \sigma_{\bbeta}(\bx) = \exp(\bbeta^\tr \bx)$. A dagger marks the true value of a quantity: $\bmeta_\dagger = (\ad, \bd) \in \Eta$ is the parameter that generated the data, whereas $\bmeta = (\alpha, \bbeta)$ denotes an arbitrary element of $\Eta$. We use $\pr_{\bmeta_\dagger}$ and $\ex_{\bmeta_\dagger}$ to highlight when a probability or expectation is taken with respect to the law underlying the data generating process.
The \emph{objective function} associated with the CMLE is the conditional log-density of the model. We take it with respect to the dominating measure
\begin{equation}
    \label{eq:dominating_measure}
    \mu(dy) := y^{-1} \, dy
    \mtext{on} \Rpos ,
\end{equation}
and not with respect to the Lebesgue measure, with the two measures being equivalent on $\Rpos$. The two choices differ only by the term $\ln y$, which does not depend on $\bmeta$, and hence both approaches lead to the same maximizer; but $\mu$ is the scale-invariant measure on $\Rpos$, hence the natural dominating measure for a scale family. Furthermore, the objective function is bounded above under $\mu$. To see this, let $g(t)=t-\exp(t)$. By \Cref{eq:pdf}, the density of $P_{\bmeta}^{(Y\mid\bX)}(\,\cdot \mid \bx)$ with respect to $\mu$ is $y \mapsto y \, f_{\bmeta}^{Y\mid \bX}(y \mid \bx)$, so that
    \begin{align}
        \label{eq:objective_function}
        m(\bmeta \mid y,\bx)
        & := \ln \rbr{y \, f_{\bmeta}^{Y\mid \bX}(y \mid \bx)}
        = \ln \alpha + g\rbr{\alpha\rbr{\bbeta^\tr\bx-\ln y}} .
    \end{align}

Two elementary bounds on $g$ are used repeatedly, here and in \Cref{sec:consistency}:
\begin{equation}
    \label{eq:g_bounds}
    g(t) \leq -1
    \mtext{and}
    g(t) \leq -\abr{t} ,
    \qquad t \in \reals .
\end{equation}
By the first bound in \Cref{eq:g_bounds}, the objective function is bounded above,
\begin{equation}
    \label{eq:m_upper_bound}
    m(\bmeta \mid y, \bx) \leq \ln \alpha - 1 ,
    \qquad (y, \bx) \in \Rpos \times \Rp .
\end{equation}
It follows that $\ex_{\bmeta_\dagger}\sbr{m(\bmeta \mid Y, \bX)}$ is well defined in $[-\infty, \ln \alpha - 1]$ for every $\bmeta \in \Eta$ and with no assumptions on the distribution of $\bX$. In general, the log-density with respect to the Lebesgue measure admits no such bound.

Averaging $m$ over a sample of size $n$ yields the criterion function to be maximized. Write $(\by, \oX) \in \Rpos^n \times \reals^{n \times p}$ for a realization of the sample, with $\by = (y_1, \ldots, y_n)^\tr$, $\oX = (\bx_1, \ldots, \bx_n)^\tr$ and $\y = (\ln y_1, \ldots, \ln y_n)^\tr$, and put
\begin{equation}
    \label{eq:n_objective_function}
    \tm[n](\bmeta \mid \by, \oX) = \frac1n \sum_{i=1}^n m(\bmeta \mid y_i, \bx_i),
    \qquad \bmeta \in \Eta.
\end{equation}
At the random sample, with $\bY := (Y_1, \ldots, Y_n)^\tr$ and $\mX := (\bX_1, \ldots, \bX_n)^\tr$, we write
\[
    M_n(\bmeta) := \tm[n](\bmeta \mid \bY, \mX)
    = \frac1n \sum_{i=1}^n m(\bmeta \mid Y_i, \bX_i) ,
    \qquad \bmeta \in \Eta ,
\]
for the resulting random function. Maximizing $M_n$ is the same as maximizing the conditional log-likelihood or the conditional log-likelihood ratio.

\subsection{Identifiability}
    \label{subsec:identifiability}
We work in a random design and impose the following condition on the covariates $\bX_i$:

\begin{condition}
    \label{cond:identifiability}
     The distribution of \(\bX_i\) is not concentrated on any subspace of $\Rp$, i.e.,
    \[
        \pr[\ba^\tr \bX_i = 0] < 1, \quad \forall \; \ba \in  \Rp \setminus \!\cbr{\0}.
    \]
\end{condition}

\Cref{cond:identifiability} ensures identifiability of the model parameter $\bmeta$.

\begin{proposition}[Identifiability]
    \label{prop:identifiability}
    Consider the model in \Cref{eq:model} and let $\bmeta_a = (\alpha_a, \bbeta_a)$ and $\bmeta_b = (\alpha_b, \bbeta_b)$ be elements of $\Eta$. The following two statements are equivalent:
    \begin{enumerate}[label=(\roman*)]
        \item For almost every realization of $\bX$, $$f_{\bmeta_a}^{Y\mid \bX}(\,\cdot \mid \bX) = f_{\bmeta_b}^{Y\mid \bX}(\,\cdot \mid \bX) \quad \mu\text{-a.e.};$$
        \item $\alpha_a = \alpha_b$ and $\rbr{\bbeta_a - \bbeta_b}^\tr \bX = 0$ almost surely.
    \end{enumerate}
    Consequently, the parameter $\bmeta$ is identifiable if and only if \Cref{cond:identifiability} holds.
\end{proposition}

A regression model typically includes an intercept, that is, the first coordinate of $\bX$ is deterministic and equal to one. Write $\bX = (1, \bX_{-1})^\tr$, with $\bX_{-1}$ the vector of the $p-1$ remaining covariates. \Cref{cond:identifiability} then turns into an \emph{affine} condition on $\bX_{-1}$: it holds if and only if the distribution of $\bX_{-1}$ is not concentrated on any affine hyperplane of $\reals^{p-1}$.

\subsection{Existence and Uniqueness of the CMLE}
    \label{subsec:cmle}

We show that the criterion function of \Cref{eq:n_objective_function} attains its maximum and that the maximum is unique under non-restrictive assumptions on the sample.
For $\alpha \in \Rpos$ and $\bb \in \Rp$, define
\begin{equation}
    \label{eq:psi}
    \psi(\alpha, \bb \mid \by, \oX)
    := \tm[n]\rbr{(\alpha, \bb / \alpha) \mid \by, \oX}
    = \ln \alpha + \frac1n \sum_{i=1}^n g\rbr{\bb^\tr \bx_i - \alpha \ln y_i}.
\end{equation}

\begin{proposition}[Concavity, existence and uniqueness of the CMLE]
    \label{prop:existence_CMLE}
    Let $(\by, \oX) \in \Rpos^n \times \reals^{n \times p}$ be a realization of the sample and let $\tm[n]$ and $\psi$ be as in \Cref{eq:n_objective_function,eq:psi}. Consider the hypotheses
    \begin{enumerate}[label=(\alph*)]
        \item $\rank(\oX) = p$;
        \item $\y \notin \operatorname{col}(\oX)$.
    \end{enumerate}
    Then the following hold.
    \begin{enumerate}[label=(\roman*)]
        \item The function $\psi(\cdot, \cdot \mid \by, \oX)$ is concave on $\Rpos \times \Rp$; that is, the conditional log-likelihood is concave in the parametrization $(\alpha, \alpha \bbeta)$. Under~(a) it is strictly concave.
        \item Under~(b), the function $\tm[n](\cdot \mid \by, \oX)$ attains its supremum on $\Eta$, and its set of maximizers is
        \[
            \cbr{\hat\alpha} \times \rbr{\hat\bbeta + \ker(\oX)}
        \]
        for some $\hat\alpha \in \Rpos$ and $\hat\bbeta \in \Rp$, where $\ker(\oX) = \cbr{\bv \in \Rp : \oX \bv = \0}$.
        \item Under~(a) and~(b), the maximizer is unique, so that the CMLE
        \[
            \mle = \argmax_{\bmeta \in \Eta} \tm[n](\bmeta \mid \by, \oX)
        \]
        is well defined.
    \end{enumerate}
\end{proposition}

Condition~(b) forces $\operatorname{col}(\oX) \neq \reals^n$ and hence $\rank(\oX) < n$; together with~(a) it therefore implies $n \geq p+1$. Without it the supremum may be infinite: if $\y = \oX \bbeta_0$ for some $\bbeta_0 \in \Rp$, then $\bbeta_0^\tr \bx_i = \ln y_i$ for every $i$, so that
$
    \tm[n]\rbr{(\alpha, \bbeta_0) \mid \by, \oX} = \ln \alpha - 1 \to + \infty,
$
as $\alpha \to +\infty$.
Without~(a) a maximizer still exists,
but $\ker(\oX)$ is then nontrivial and the maximizer is no longer unique.

Part~(i)
says that the calculation of the CMLE
is a convex program in the coordinates $(\alpha, \alpha \bbeta)$, so that there are no spurious local maxima, every stationary point is the global maximum, and a solver may be started anywhere.

Both hypotheses~(a) and~(b) are asymptotically free.
For later use, we state the following two properties for a generic sample size $k$.

\begin{lemma}[Design regularity]
    \label{lem:design_regularity}
    Let $\cbr{(Y_i, \bX_i)}_{i=1}^k$ be an iid sample such that $Y \mid \bX = \bx$ is distributed as in \Cref{eq:model}, and write
    \[
        \mX := \rbr{\bX_1, \ldots, \bX_k}^\tr \in \reals^{k \times p}
        \mtext{and}
        \Y := \rbr{\ln Y_1, \ldots, \ln Y_k}^\tr \in \reals^k .
    \]
    \begin{enumerate}[label=(\roman*)]
        \item If $k \geq p+1$, then
        \[
            \pr_{\bmeta_\dagger}\sbr{\Y \in \operatorname{col}(\mX)} = 0 ,
        \]
        whatever the distribution of $\bX$.
        \item If \Cref{cond:identifiability} is satisfied, then, almost surely, $\rank(\mX) = p$ for all $k$ large enough. In particular, $\pr\sbr{\rank(\mX) = p} \to 1$ as $k \to \infty$.
    \end{enumerate}
\end{lemma}

In the terms of \Cref{prop:existence_CMLE}, hypothesis~(b) thus holds with probability one for every $k \geq p+1$ and hypothesis~(a) with probability tending to one; equivalently, the $k \times (p+1)$ matrix $(\Y, \mX)$ has full rank with probability tending to one.

This settles the definition of the estimator, which we record here once and for all. Take $k = n$, the size of the whole sample. By part~(i) of \Cref{lem:design_regularity}, hypothesis~(b) of \Cref{prop:existence_CMLE} holds with probability one for every $n \geq p+1$, and part~(ii) of that proposition then already provides a maximizer of $M_n$ over $\Eta$, the set of them being $\cbr{\hat\alpha} \times \rbr{\hat\bbeta + \ker(\mX)}$. Hypothesis~(a) holds with probability tending to one, by part~(ii) of \Cref{lem:design_regularity}, and reduces that set to a single point. On the complementary event we take for $\mle$ the maximizer of minimal Euclidean norm.
With this convention $\mle$ is defined for every $n \geq p+1$, almost surely, and, by construction, it satisfies
\[
    M_n(\mle) \geq M_n(\bmeta) , \qquad \bmeta \in \Eta.
\]

\subsection{The Asymptotic Criterion Function}
    \label{subsec:asymptotic_criterion}

By \Cref{eq:m_upper_bound}, the positive part of $m(\bmeta \mid Y, \bX)$ is bounded by the constant $\rbr{\ln \alpha - 1}_+$, so that $m(\bmeta \mid Y, \bX)$ is quasi-integrable for every $\bmeta \in \Eta$.
Recall that a random variable $H$ is called \emph{quasi-integrable} if at least one of $\ex\sbr{H_+}$ and $\ex\sbr{H_-}$ is finite, where $H_+ = \max(H, 0)$ and $H_- = \max(-H, 0)$; in that case $\ex\sbr{H} := \ex\sbr{H_+} - \ex\sbr{H_-}$ is well defined in $[-\infty, +\infty]$.

By the law of large numbers for quasi-integrable summands, we have, almost surely,
\begin{equation}
    \label{eq:asymptotic_criterion_function}
    M_n(\bmeta) \xrightarrow{n \to \infty} M(\bmeta) := \ex_{\bmeta_\dagger}\sbr{m(\bmeta \mid Y, \bX)} \leq \ln \alpha - 1 ,
\end{equation}
the limit being well defined in $[-\infty, \ln \alpha - 1]$ for every $\bmeta \in \Eta$.
At the true parameter, that limit is finite and free of the covariates. By \Cref{eq:model}, we have $\ln Y = \bd^\tr \bX + \ad^{-1} G$ with $G$ standard Gumbel and independent of $\bX$, so that,
by \Cref{eq:objective_function},
\begin{equation}
    \label{eq:m_at_truth}
    m(\bmeta_\dagger \mid Y, \bX) = \ln \ad + g(-G) = \ln \ad - G - \exp(-G),
    \mtext{so that} M(\bmeta_\dagger) = \ln \ad - \gamma - 1 ,
\end{equation}
because $\ex\sbr{G} = \gamma$, the Euler--Mascheroni constant, and $\exp(-G) \sim \Exp(1)$. The law of $m(\bmeta_\dagger \mid Y, \bX)$ depends neither on the distribution of $\bX$ nor on $\bd$. Since $M(\bmeta_\dagger)$ is finite, it may be subtracted from $M(\bmeta)$,
leaving the expected log-likelihood ratio:
\begin{equation}
    \label{eq:M_difference}
    M(\bmeta) - M(\bmeta_\dagger)
    = \ex_{\bmeta_\dagger}\sbr{\ln \frac{f_{\bmeta}^{Y\mid \bX}}{f_{\bmeta_\dagger}^{Y\mid \bX}}} ,
    \qquad \bmeta \in \Eta .
\end{equation}
The set of points of maximum of the asymptotic criterion function is
\begin{equation}
    \label{eq:Eta_0}
    \Eta_0 := \cbr{\bmeta_0 \in \Eta : M(\bmeta_0) = \sup_{\bmeta \in \Eta} M(\bmeta)}.
\end{equation}

\begin{lemma}[Unique maximizer]
    \label{lem:unique_maximizer}
    Consider the model given by \Cref{eq:model}, and assume that \Cref{cond:identifiability} is satisfied. Then
    $\Eta_0 = \cbr{\bmeta_\dagger}$.
\end{lemma}

\section{An Extension of Wald's Consistency Theorem}
\label{sec:wald}

    The goal of this section is to present, in an abstract setting, an extended version of Wald's consistency theorem \citep{Wald1949}, in the form given as Theorem~5.14 in \citet{van2000asymptotic}. This theorem will play an important role in proving consistency of the CMLE associated with \Cref{eq:model}.

    Let $Z_1, Z_2, \ldots$ be an iid sequence of random elements with values in a measurable space $\rbr{\cZ, \cA}$, with distribution $P_Z$, and write $\bZ^{(l)} := \rbr{Z_1, \ldots, Z_l}$, for $l \in \naturals$, for the block formed by the first $l$ of them. Let $(\Theta, d)$ be a metric space, let $\sTheta$ be a nonempty subset of it, and define the objective function $m : \sTheta \times \cZ \to \reals \cup \cbr{-\infty}$, with $z \mapsto \mm{\theta}{z} := m(\theta, z)$ measurable for every $\theta \in \sTheta$. Let the associated random criterion function be defined as $\theta \mapsto M_n(\theta) := n^{-1}\sum_{i=1}^n \mm{\theta}{Z_i}$.

For $\theta \in \Theta$ and $\rho > 0$, let $B(\theta, \rho)$ denote the open ball $\cbr{\theta' \in \Theta : d(\theta', \theta) < \rho}$.

    \begin{definition}[Upper-semicontinuity]
        \label{def:upper_semicontinuity}
        Let $W$ be a random element in a measurable space $\cW$, with distribution $P_W$, and let $\varphi : \Theta \times \cW \to \reals \cup \cbr{-\infty}$ be such that $w \mapsto \varphi_\theta(w) := \varphi(\theta, w)$ is measurable for every $\theta \in \Theta$. The random map $\theta \mapsto \varphi_\theta(W)$ is \emph{almost surely upper-semicontinuous at} $\theta \in \Theta$ if there exists a set $\cN_\theta \subseteq \cW$ with $P_W\sbr{\cN_\theta} = 0$ such that
        \[
            \limsup_{j \to \infty} \varphi_{\theta_j}(w) \leq \varphi_\theta(w)
        \]
        for every $w \in \cW \setminus \cN_\theta$ and every sequence $\theta_j \to \theta$ in $\Theta$. It is \emph{almost surely upper-semicontinuous} if it is so at every $\theta \in \Theta$.
    \end{definition}

    \begin{condition}[Wald's integrability condition]
        \label{cond:wald_integrability}
        In the setting of \Cref{def:upper_semicontinuity}, the random map $\theta \mapsto \varphi_\theta(W)$ is said to \emph{satisfy Wald's integrability condition at} $\theta \in \Theta$ if there exists $\delta > 0$ such that
        \begin{enumerate}[label=(\alph*)]
            \item $\sup_{\theta' \in B(\theta, \rho)} \varphi_{\theta'}(W)$ is a $[-\infty,\infty]$-valued random variable for every radius $\rho \in (0, \delta]$, and
            \item its positive part is integrable at the largest of those radii,
            \begin{equation*}
                \ex\sbr{\rbr{\sup_{\theta' \in B(\theta, \delta)} \varphi_{\theta'}(W)}_+} < \infty.
            \end{equation*}
        \end{enumerate}
        It is said to \emph{satisfy Wald's integrability condition} if it does so at every $\theta \in \Theta$.
    \end{condition}
    \noindent Item~(b) at the radius $\delta$ implies the same bound at every smaller radius, a supremum over a smaller ball being smaller; item~(a), on the other hand, is a genuine requirement at each radius, a supremum over an uncountable index set not being measurable by itself.

    \begin{theorem}
    \label{thm:wald_blocks}
        In the setting described above, fix $k \in \naturals$ and let, for every $l \in \cbr{k, \ldots, 2k-1}$,
        \[
            \om : \Theta \times \cZ^l \to \reals \cup \cbr{-\infty}
        \]
        be such that $\bz^{(l)} \mapsto \om\rbr{\theta \mid \bz^{(l)}}$ is measurable for every $\theta \in \Theta$ and such that
        \begin{enumerate}[label=(\roman*)]
            \item for every $\theta \in \sTheta$ and every $\bz^{(l)} = \rbr{z_1, \ldots, z_l} \in \cZ^l$,
            \begin{equation}
                \label{eq:restriction_property}
                \om\rbr{\theta \mid \bz^{(l)}} = \frac1l \sum_{i=1}^l \mm{\theta}{z_i} ;
            \end{equation}
            \item the random map $\theta \mapsto \om\rbr{\theta \mid \bZ^{(l)}}$ on $\Theta$ is almost surely upper-semicontinuous (\Cref{def:upper_semicontinuity}) and satisfies Wald's integrability condition (\Cref{cond:wald_integrability}), both with $\cW = \cZ^l$ and $W = \bZ^{(l)}$;
            \item $\ex[\om[k](\theta \mid \bZ^{(k)})] = -\infty$ for every $\theta \in \Theta \setminus \sTheta$.
        \end{enumerate}
        Let
        \begin{equation}
            \label{eq:Theta_0_blocks}
            \mTheta := \cbr{\theta_0 \in \sTheta : \ex\sbr{\mm{\theta_0}{Z}} = \sup_{\theta \in \sTheta} \ex\sbr{\m Z}}
        \end{equation}
        denote the set of points of maximum. Assume that this set is nonempty. Then for any $\sTheta$-valued estimator sequence $\est_n$ such that $M_n(\est_n) \geq M_n(\theta_0) - o_{\pr}(1)$ for some $\theta_0 \in \mTheta$, we have, for every $\eps > 0$ and every compact set $K \subseteq \Theta$,
        \begin{equation*}
            \pr\sbr{d\rbr{\est_n, \mTheta} \geq \eps \; \text{ and } \; \est_n \in K} \to 0, \mtext{as} n \to \infty,
        \end{equation*}
        where $d\rbr{\est_n, \mTheta} := \inf_{\theta \in \mTheta} d\rbr{\est_n, \theta}$.
    \end{theorem}

    Two remarks:
   \begin{itemize}
       \item The expectations in \Cref{eq:Theta_0_blocks} are well defined in $[-\infty, \infty)$. On $\sTheta$, assumption~(i) turns the Wald integrability of assumption~(ii) into a statement about a block average, so that \Cref{lem:summand_integrability} in the appendix applies and yields $\ex[\rbr{\m Z}_+] < \infty$ for every $\theta \in \sTheta$ (\Cref{eq:single_observation_integrability}). The expectation $\ex\sbr{\m Z}$ is therefore well defined in $[-\infty, \infty)$, which implies $\ex[\om(\theta \mid \bZ^{(l)})] = \ex\sbr{\m Z}$ for all block sizes $l$.

        \item Assumption~(iii) is what allows the set of points of maximum to be written in terms of a single observation. Provided the supremum in \Cref{eq:Theta_0_blocks} is larger than $-\infty$, that assumption makes $\mTheta$ at the same time the set of maximizers of $\theta \mapsto \ex[\om[k](\theta \mid \bZ^{(k)})]$ over the whole of $\Theta$, which is the form in which the proof uses it: the points outside $\sTheta$ are excluded by assumption~(iii) itself.
   \end{itemize}

    Note that \Cref{thm:wald_blocks} reduces to the original version of Wald's theorem when $\sTheta=\Theta$ and $k=1$. Thus, our theorem extends Wald's consistency theorem to objective functions that may fail to be almost surely upper-semicontinuous or to satisfy Wald's integrability condition, as long as these conditions hold for a suitably extended block average.

\section{Consistency}
    \label{sec:consistency}
    
In the setting of \Cref{sec:framework}, we aim to prove consistency of the CMLE. Although several results establish consistency of the CMLE under fairly general assumptions (see Chapter 5 of \citealp{van2000asymptotic} and Chapters 12--13 of \citealp{Wooldridge2010}), these results are based on the assumption that the parameter space is compact. \citet{BucherSegersStaud2025} obtain consistency of M-estimators for a class of models more general than ours, in the sense that $Y \mid \bX$ is only assumed to be in the domain of attraction of the GEV distribution rather than exactly Fréchet; they too work on a compact parameter space, and they require in addition that the covariates take values in a compact set. However, we find it unsatisfactory to impose compactness by simply restricting the parameter space to a compact subset of $\Eta$. This is often undesirable as an assumption because it is not verifiable from the observed data and may impose arbitrary bounds on parameters that are naturally unrestricted. One must then either apply Wald's argument on a suitable compactification of $\Eta$, or show that there exists a compact set $K\subset\Eta$ containing the CMLE with probability tending to one.

We follow the first approach and work with a compactification of the whole parameter space. That is, we embed $\Eta$ into a compact metrizable space $\overline\Eta$ by adjoining boundary points corresponding to sequences escaping every compact subset of $\Eta$. We then extend the objective function ($m$) to $\overline\Eta$ and verify the conditions of \Cref{thm:wald_blocks}. In particular, this requires finding a compactification for which the extended objective function is upper-semicontinuous in the sense of \Cref{def:upper_semicontinuity} and satisfies \Cref{cond:wald_integrability} for some $k$.

We now elaborate on a specific compactification of $\Eta$ under which the conditions of \Cref{thm:wald_blocks} are satisfied for large enough $k$.

\subsection{Compactification}
\label{subsec:compactification}

In the course of the project, we considered several compactifications of the parameter space and found that, ultimately, none of them worked with the original version of Wald's consistency theorem. The compactification below proved valid under the extension in \Cref{thm:wald_blocks}, although this does not rule out the possibility that some other compactification might satisfy the conditions of the original theorem.

Recall that $\bmeta = \rbr{\alpha, \bbeta} \in \Rpos \times \Rp = \Eta$. The compactification is constructed as follows:
\begin{itemize}
    \item The coordinate $\alpha \in \Rpos$ is extended to the compact interval $[0,\infty]$ by adjoining the endpoints $0$ and $\infty$;
    \item For $\bbeta \in \Rp$, we pass to the \emph{cosmic closure} of $\Rp$, denoted by $\csm(\Rp)$ \citep[Section~3.A]{rockafellar2009variational}. In this construction, which we explain below, each direction in $\Rp$ is identified with a point at infinity, and the collection of all such points forms the horizon of the space.
\end{itemize}
Let $\dir \bbeta := \cbr{\lambda \bbeta : \lambda \in \Rpos}$ denote the \emph{direction} of $\bbeta \in \Rp \setminus \cbr{\0}$. The direction of the origin is left undefined. For $\bbeta_a,\bbeta_b \in \Rp \setminus \cbr{\0}$, we have
\[
\dir \bbeta_a = \dir \bbeta_b
\quad\Longleftrightarrow\quad
\bbeta_a=\lambda \bbeta_b,
\ \text{for some }\lambda>0.
\]
The \emph{horizon} of $\Rp$, denoted by
\[
    \hzn(\Rp) := \cbr{ \dir \bbeta : \bbeta \in \Rp \setminus \cbr{\0}},
\]
is the set of all such direction points. The \emph{cosmic closure} of $\Rp$ is defined as
\[
\csm(\Rp):=\Rp\cup \hzn(\Rp).
\]

A sequence $\bbeta_k\in\Rp$ is said to converge to a horizon point $\dir(\bbeta)\in\hzn(\Rp)$ if there exists a sequence of scalars $\lambda_k\downarrow 0$ such that
\[
\lambda_k \bbeta_k \to \bbeta
\qquad\text{as }k\to\infty.
\]
Likewise, for horizon points, $\dir \bbeta_k \to \dir\bbeta$ if there exist scalars $\lambda_k>0$ such that
\[
\lambda_k \bbeta_k \to \bbeta
\qquad\text{as }k\to\infty.
\]

\begin{figure}
    \centering
    \begin{tikzpicture}[scale=1.1]

    \begin{pgfonlayer}{background}
        \fill[gray!10]
            (-1.5,-2.5) -- (6,-2.5) -- (4,0.5) -- (-3,0.5) -- cycle;

        \draw
            (-1.5,-2.5) -- (6,-2.5) -- (4,0.5) -- (-3,0.5) -- cycle;
    \end{pgfonlayer}

    \node at (-1.25,-2.25) {$\Rp$};

    \fill[white]
    (-2,1)
    .. controls (-2,-1.1046) and (-0.1046,-1) ..
    (0,-1.0)
    .. controls (0.1046,-1) and (2,-1.1046) ..
    (2,1)
    -- cycle;

    \draw
    (-2,1)
    .. controls (-2,-1.1046) and (-0.1046,-1) ..
    (0,-1.0)
    .. controls (0.1046,-1) and (2,-1.1046) ..
    (2,1);

    \fill (0,-0.5) circle (1pt);
    \node[left] at (0,-0.5) {$(0,0)$};

    \fill (0,1) circle (1pt);
    \node[left] at (0,1) {$(0,1)$};

    \draw (-2,1) arc (180:360:2 and 0.7);
    \draw (-2,1) arc (180:0:2 and 0.7);

    \draw[dotted] (0,-0.5) -- (0,1);
    \draw[->] (0,1) -- (0,3) node[left] {$\reals$};

    \filldraw[draw=blue, pattern=north west lines, pattern color=blue, opacity=0.7]
    plot[smooth cycle,tension=0.9] coordinates {
      (0.5,-0.75)
      (0.435,-0.25)
      (0.65,0.05)
      (1.1,0)
      (1.4,0.47)
      (1.7,0.375)
      (1.4,-0.25)
      (0.95, -0.65)
    };

    \begin{scope}
        \clip
        (-1.5,-2.5) -- (6,-2.5) -- (4,0.5) -- (-3,0.5) -- cycle;

        \filldraw[draw=blue, pattern=north east lines, pattern color=blue, opacity=0.4]
        plot[smooth cycle,tension=0.9] coordinates {
          (1.22,-1.56)
          (1.25,-2.20)
          (2.84,-2.31)
          (3.85,-1.92)
          (5.85,-2.42)
          (6.62,-2.37)
          (5.07,-1.62)
        };
    \end{scope}

    \node[right] at (4.5,-1.35) {$C$};

    \coordinate (A) at (0,1);

    \coordinate (B1) at (0.7,-1.80);
    \coordinate (M1) at ($(A)!0.575!(B1)$);
    \draw[dotted] (A) -- (M1);
    \draw[dashed] (M1) -- (B1);
    \fill[orange] (M1) circle (1pt);
    \fill[orange] (B1) circle (1pt);
    \coordinate (B2) at (2.34,-2.31);
    \coordinate (M2) at ($(A)!0.29!(B2)$);
    \draw[dotted] (A) -- (M2);
    \draw[dashed] (M2) -- (B2);
    \fill[orange] (M2) circle (1pt);
    \fill[orange] (B2) circle (1pt);
    \coordinate (B3) at (3.55,-1.92);
    \coordinate (M3) at ($(A)!0.325!(B3)$);
    \draw[dotted] (A) -- (M3);
    \draw[dashed] (M3) -- (B3);
    \fill[orange] (M3) circle (1pt);
    \fill[orange] (B3) circle (1pt);
    \coordinate (B4) at (5.5,-0.92);
    \coordinate (M4) at ($(A)!0.265!(B4)$);
    \draw[dotted] (A) -- (M4);
    \draw[dashed] (M4) -- (B4);
    \fill[orange] (M4) circle (1pt);
    \coordinate (B5) at (4.2,-1.525);
    \coordinate (M5) at ($(A)!0.377!(B5)$);
    \draw[dotted] (A) -- (M5);
    \draw[dashed] (M5) -- (B5);
    \fill[orange] (M5) circle (1pt);
    \fill[orange] (B5) circle (1pt);

    \node[below] at (3.8,-1.9) {$(\bbeta,0)$};
\end{tikzpicture}
    \caption{Hemispherical model associated with the compactification of $\Rp$. Based on Figure 3--2 in \citet{rockafellar2009variational}.}
    \label{fig:hemispherical_model}
\end{figure}
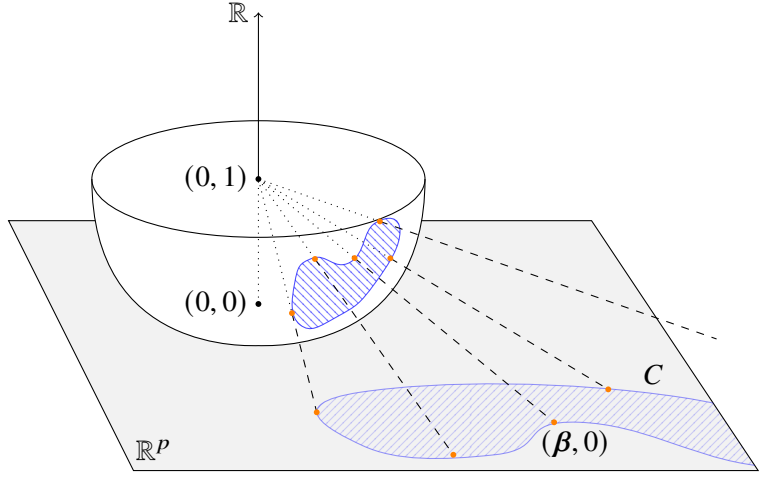

The hemispherical model provides a useful visualization (see \Cref{fig:hemispherical_model}). Consider the unit hemisphere
\[
H_p := \cbr{(\bx, h) \in \Rp \times \reals : h \leq 1,\ \norm{\bx}^2 + (h-1)^2 = 1},
\]
together with the hyperplane $\Rp \times \cbr{0}$ tangent to the hemisphere at $(\0,0)$. Identify the hyperplane with $\Rp$ itself. There is a one-to-one correspondence between the points of $\Rp$ and the points of the open hemisphere: a point $\bbeta \in \Rp$ corresponds to the unique point at which the ray emanating from the center $(\0, 1)$ of the hemisphere and passing through $(\bbeta, 0)$ intersects the open hemisphere. The infinity points are represented by the rim of $H_p$. Since $H_p$ is compact, this yields a compact representation of $\csm(\Rp)$ (see Theorem~3.2 in \citealp{rockafellar2009variational}). In other words, we have a homeomorphism that maps the points on the hyperplane to the interior of the hemisphere, and the directions representing points on the horizon to the points on the rim of the hemisphere.

Consequently, the extended parameter space is
\[
\overline\Eta :=[0,\infty]\times \csm(\Rp).
\]
Endow $[0,\infty]$ and $\csm(\Rp)$ with metrics $d_{\mathrm A}$ and $d_{\mathrm B}$ that induce their respective topologies, and define the metric
\[
d_\infty(\bmeta_1,\bmeta_2)
:=
\max\bigl\{d_{\mathrm A}(\alpha,\alpha'),\, d_{\mathrm B}(\bbeta,\bbeta')\bigr\},
\]
for $\bmeta_1=(\alpha,\bbeta)$ and $\bmeta_2=(\alpha',\bbeta')$ in $\overline\Eta$. Then $(\overline\Eta,d_\infty)$ is a compact metric space, as a product of compact metric spaces.

The parameter space $\Eta$ is an open subset of $\overline\Eta$ on which $d_\infty$ induces the Euclidean topology.
For a fixed $\bmeta_0 \in \Eta$ and for $\Eta$-valued random vectors $\bmeta_n$,
\begin{equation}
    \label{eq:metrics_agree}
    d_\infty\rbr{\bmeta_n, \bmeta_0} \pto 0
    \qquad \Longleftrightarrow \qquad
    \norm{\bmeta_n - \bmeta_0} \pto 0 .
\end{equation}
Within the original parameter set $\Eta$, convergence in probability on the compactification on $\overline \Eta$ and on $\Eta$ are thus equivalent.

    \subsection{Consistency of the CMLE}
    To show consistency of the CMLE under the model in \Cref{eq:model}, it suffices to ensure that $\bmeta_\dagger$ is identifiable and a unique point of maximum for $M(\bmeta) = \ex_{\bmeta_\dagger}\sbr{m\rbr{\bmeta \mid Y, \bX}}$ of \Cref{eq:asymptotic_criterion_function}, and to verify the conditions of \Cref{thm:wald_blocks} for the CMLE on the compactified parameter space $\overline \Eta$ defined in \Cref{subsec:compactification}, using a suitable extension of the objective function $m$ over blocks of observations. The first half is already done: the set $\Eta_0$ of \Cref{eq:Eta_0} is \Cref{eq:Theta_0_blocks} instantiated at this $m$, and \Cref{lem:unique_maximizer} identifies it as the singleton $\cbr{\bmeta_\dagger}$. What remains is the extension, and the near-maximizer hypothesis on the estimator, which is a hypothesis of \Cref{thm:consistency} below.

    Define $\tm$ as
    \begin{equation}
        \label{eq:l_objective_function}
        \tm(\bmeta \mid \by,  \oX) := \frac1l\sum_{i=1}^l m(\bmeta \mid y_i, \bx_i),
    \end{equation}
    for $l$ iid observations. We are interested in a measurable extension of $\tm$, which we will denote $\om$, to the compactified parameter space such that the following conditions hold for all $l \in \cbr{k, \ldots, 2k -1}$ for some $k \in \naturals$:
    \begin{enumerate}
        \item $\om$ is almost surely upper-semicontinuous (\Cref{def:upper_semicontinuity});
        \item $\om$ satisfies Wald's integrability condition (\Cref{cond:wald_integrability});
        \item The expectation of $\om[k]$ on points $\bmeta$ not in $\Eta$ is $-\infty$.
    \end{enumerate}

    \begin{lemma}[Upper-semicontinuity]
        \label{lem:USC}
        Fix $l \in \naturals$ and let $(\by,\oX) \in \Rpos^l \times \reals^{l \times p}$ be a realization of $l$ observations, with $\by = (y_1,\ldots,y_l)^\tr$ and $\oX = (\bx_1, \ldots, \bx_l)^\tr$, $\bx_i \in \Rp$, such that
        \[
            \y \notin \operatorname{col}(\oX) ,
            \mtext{where} \y := (\ln y_1, \ldots, \ln y_l)^\tr .
        \]

        For every horizon point $\bgamma=\dir\bbeta\in \hzn(\Rp)$, define
        \[
        \cM_\bgamma:=\cbr{\bx \in \Rp : \bbeta^\tr \bx = 0}.
        \]
        Then the map $\om(\cdot\mid \by,\oX):\overline\Eta\to \reals\cup\{-\infty\}$ defined as
        \[
        \om(\bmeta\mid \by,\oX) :=
        \begin{dcases}
            \tm(\bmeta\mid \by,\oX)
            & \text{if } \bmeta=(\alpha,\bbeta)\in \Rpos \times \Rp,\\
            \sup_{\bb\in\Rp}\tm((\alpha,\bb)\mid \by, \oX)
            & \begin{aligned}[t]
                & \text{if } \bmeta=(\alpha,\bgamma) \in \Rpos \times \hzn(\Rp) \\
                & \qquad \text{and } \cbr{\bx_1, \ldots, \bx_l} \subseteq \cM_\bgamma,
            \end{aligned} \\
            -\infty
            & \text{otherwise,}
        \end{dcases}
        \]
        is upper-semicontinuous on $\overline\Eta$.
    \end{lemma}

    In \Cref{lem:integrable_boundary,lem:integrability}, $(\bY, \mX)$ denotes the random counterpart of the realization $(\by, \oX)$ of \Cref{lem:USC}: for a block $(Y_1, \bX_1), \ldots, (Y_l, \bX_l)$ of $l$ observations,
    \begin{equation}
        \label{eq:block_notation}
        \bY := \rbr{Y_1, \ldots, Y_l}^\tr \in \Rpos^l
        \mtext{and}
        \mX := \rbr{\bX_1, \ldots, \bX_l}^\tr \in \reals^{l \times p} .
    \end{equation}
    We write
    \begin{equation}
        \label{eq:null_set}
        \cN_\infty := \cbr{(\by, \oX) \in \Rpos^l \times \reals^{l \times p} : \y \in \operatorname{col}(\oX)}
    \end{equation}
    for the set of realizations excluded by \Cref{lem:USC}. By part~(i) of \Cref{lem:design_regularity}, applied to a block of $l \geq p+1$ observations,
    \begin{equation}
        \label{eq:null_set_is_null}
        \pr_{\bmeta_\dagger}\sbr{(\bY, \mX) \in \cN_\infty} = 0 ,
    \end{equation}
    whatever the distribution of $\bX$.

    \begin{lemma}
        \label{lem:integrable_boundary}
        Let $\{(Y_i, \bX_i)\}_{i = 1}^l$ be iid and such that \Cref{eq:model} holds, with fixed $l \geq p+1$. Assume that \Cref{cond:identifiability} is satisfied, and define $\om$ as in \Cref{lem:USC}. Then, for all $\bmeta \in \overline \Eta \setminus \Eta$, we have
        \[
        \ex_{\bmeta_\dagger}\sbr{\om(\bmeta \mid \bY,\mX)} = -\infty.
        \]
    \end{lemma}
    Before the integrability itself, the local suprema of $\om$ have to be shown to be random variables at all, which is item~(a) of \Cref{cond:wald_integrability}.

    \begin{lemma}[Measurability]
        \label{lem:measurability}
        Fix $l \in \naturals$ and let $\om$ be as in \Cref{lem:USC}. Then, with the convention $\sup \emptyset = -\infty$,
        \begin{enumerate}[label=(\roman*)]
            \item for every $\bmeta \in \overline\Eta$, the map $(\by, \oX) \mapsto \om(\bmeta \mid \by, \oX)$ is Borel measurable on $\Rpos^l \times \reals^{l \times p}$;
            \item for every non-empty open set $U \subseteq \overline\Eta$, the map
            \[
                \Rpos^l \times \reals^{l \times p} \to [-\infty, +\infty] :
                (\by, \oX) \mapsto \sup_{\bmeta \in U} \om(\bmeta \mid \by, \oX)
            \]
            is Borel measurable.
        \end{enumerate}
    \end{lemma}

    The next result supplies item~(b) of \Cref{cond:wald_integrability}, in a stronger form than
    that item asks for: the supremum is over the whole of $\overline\Eta$ rather than over a ball.
    By part~(ii) of \Cref{lem:measurability}, applied to the open set $U = \overline\Eta$, that
    supremum is a random variable.

    \begin{lemma}[Integrability of the global supremum]
        \label{lem:integrability}
        Let $\cbr{(Y_i, \bX_i)}_{i=1}^l$ be iid and such that \Cref{eq:model} holds, and let $(\bY, \mX)$ be as in \Cref{eq:block_notation}, with fixed $l \geq p+1$. Let $\om$ be as defined in \Cref{lem:USC}. Then
        \begin{equation}
            \label{eq:wald_integrability_frechet}
            \ex_{\bmeta_\dagger}\sbr{\rbr{\sup_{\bmeta \in \overline\Eta} \om\rbr{\bmeta \mid \bY, \mX}}_+} < \infty .
        \end{equation}
    \end{lemma}

    Write
        \[
            B(\bmeta, \rho) := \cbr{\bmeta' \in \overline\Eta : d_\infty(\bmeta, \bmeta') < \rho}
        \]
    for the open ball of radius $\rho$ around $\bmeta$ in the metric of \Cref{subsec:compactification}.
    Every such ball is a subset of $\overline\Eta$, and a supremum does not decrease when the set over
    which it is taken is enlarged, so \Cref{eq:wald_integrability_frechet} gives item~(b) of
    \Cref{cond:wald_integrability} at every $\bmeta \in \overline\Eta$ and with any $\delta > 0$.

    Together, \Cref{lem:measurability,lem:integrability} say that the random map $\bmeta \mapsto \om(\bmeta \mid \bY, \mX)$ satisfies Wald's integrability condition (\Cref{cond:wald_integrability}) at every point of $\overline\Eta$. We can now present the main result of this section. It covers every near-maximizer, and in particular the CMLE, whose existence in $\Eta$ for $n \ge p+1$ was already established at the end of \Cref{subsec:cmle}.

\begin{theorem}
    \label{thm:consistency}
    Let $\{(Y_i, \bX_i)\}_{i=1}^n$ be an iid sample such that $Y\mid\bX = \bx$ is distributed as in \Cref{eq:model}, with true parameter $\bmeta_\dagger \in \Eta = \Rpos \times \Rp$. Assume that \Cref{cond:identifiability} is satisfied. Let $m(\bmeta \mid y, \bx)$ be the objective function of \Cref{eq:objective_function}. Then, for any sequence of $\Eta$-valued estimators $\nme$ of $\bmeta_\dagger$ such that
    \[
        M_n(\nme) \geq M_n(\bmeta_\dagger) - o_{\pr}(1)
    \]
    as $n \to \infty$, we have
    \[
        \nme \pto \bmeta_\dagger.
    \]
\end{theorem}

\section{Asymptotic Normality}
    \label{sec:asymptotic_normality}

Once consistency has been established, asymptotic normality of the CMLE follows from classical likelihood theory. The only additional requirement is the following condition on the covariates.

\begin{condition}
    \label{cond:mgf}
    The moment generating function (MGF) of $\bX_i$ exists, i.e., we have $M_{\bX_i}(\bt) = \ex\sbr{\exp\rbr{\bt^\tr \bX_i}} < \infty$ for all $\bt \in \Rp$ in a neighborhood of $\0$.
\end{condition}

Note that \Cref{cond:identifiability} implies positive definiteness of $\ex\sbr{\bX \bX^\tr}$, since
\begin{align*}
    \ba^\tr \ex\sbr{\bX \bX^\tr}\ba & = \ex\sbr{\rbr{\ba^\tr \bX}^2} > 0,
\end{align*}
for every $\ba \neq \0$, provided the second moments of $\bX$ are finite, which follows from the existence of the MGF in \Cref{cond:mgf}. This matrix enters the Fisher information matrix and is therefore needed only for the asymptotic normality of $\mle$, not for its consistency.

The CMLE $\mle$ was defined in \Cref{subsec:cmle}. For every $n \geq p+1$ it maximizes $M_n$ over $\Eta$ almost surely. The maximizer is unique on an event whose probability tends to one. The precise definition in case the maximizer is not unique does not affect \Cref{thm:consistency} nor the limit distribution below, since the probability of that event tends to zero.

\begin{theorem}
    \label{thm:normality}
    Let $\{(Y_i, \bX_i)\}_{i=1}^n$ be an iid sample such that $Y\mid\bX = \bx$ is distributed as in \Cref{eq:model}, with $\bmeta_\dagger \in \Eta = \Rpos \times \Rp$. Assume that \Cref{cond:identifiability,cond:mgf} are satisfied. Let $m(\bmeta \mid Y, \bX)$ be the objective function of \Cref{eq:objective_function} and let $\mle$ be the CMLE, as fixed in \Cref{subsec:cmle}.
    Let $\Fi$ denote the Fisher information matrix, with
    \[
        \Fi
            =
            \begin{pmatrix}
            \dfrac{1}{\ad^2} \rbr{\dfrac{\pi^2}{6} + (1-\gamma)^2}
            & (1 - \gamma)\ex[\bX]^\tr \\
            (1 - \gamma)\ex[\bX] & \ad^2 \ex[\bX \bX^\tr] \\
            \end{pmatrix},
    \]
    where $\gamma$ is the Euler--Mascheroni constant.
    Then
    \begin{equation*}
        \sqrt n (\mle - \bmeta_\dagger)\rightsquigarrow \Normal\rbr{\0, \Fi^{-1}}.
    \end{equation*}
\end{theorem}

The information matrix has a $1 + p$ block structure, and both its off-diagonal block and its dependence on $\ad$ are informative. Write $\mle = (\hat\alpha_n, \hat\bbeta_n)$.
The two estimators $\hat\alpha_n$ and $\hat\bbeta_n$ are asymptotically independent if and only if $\ex[\bX] = \0$. A model with an intercept never has that property, but centering the remaining covariates isolates the dependence:
the tail parameter is then asymptotically correlated with the estimated intercept alone and independent of the estimated slopes. As for $\ad$, it enters $\Fi$ as a scale factor only.
The accuracy of $\hat\alpha_n$ relative to $\ad$ is thus the same whatever the tail, while that of $\hat\bbeta_n$ deteriorates as $\ad$ decreases and the tail grows heavier.

The asymptotic variance is easy to estimate, since $\Fi$ involves the unknown quantities only through $\ad$ and the first two moments of $\bX$: substitute $\hat\alpha_n$ for $\ad$ and the sample moments $n^{-1}\sum_{i=1}^n \bX_i$ and $n^{-1}\sum_{i=1}^n \bX_i \bX_i^\tr$ for $\ex[\bX]$ and $\ex[\bX \bX^\tr]$.
Together with \Cref{thm:normality} and Slutsky's lemma, this yields Wald confidence intervals for $\ad$ and for the components of $\bd$, and confidence ellipsoids for $\bmeta_\dagger$.

We invoke \Cref{cond:mgf} in the proof of \Cref{thm:normality} when deriving integrable uniform bounds on the third-order partial derivatives of $\bmeta \mapsto m(\bmeta \mid y, \bx)$ in order to verify the assumptions of Theorem~5.41 in \citet{van2000asymptotic}. With a different proof technique, a weaker moment assumption on $\bX$ may perhaps already be sufficient. In any case, finite second moments are a necessary assumption for the conclusion of the theorem to hold.

\section{Conclusion}
\label{sec:conclusion}

In this work, we have established consistency and asymptotic normality of the conditional maximum likelihood estimator for the Fréchet regression model with log-linear scale. Consistency was obtained on the full parameter space, thereby extending previous results that relied on restricting the estimator to a compact subset of $\Eta$. The key step was the compactification $\overline\Eta$ together with an extension of Wald's consistency theorem. Beyond the present Fréchet regression model, these tools provide a framework for establishing consistency in M-estimation problems with non-compact parameter spaces and objective functions that do not satisfy the usual integrability or upper-semicontinuity assumptions.

Several limitations should be kept in mind. The analysis assumes exact Fréchet specification and does not cover the broader domain-of-attraction setting. The observations are assumed to be independent and identically distributed. The scale parameter is restricted to a log-linear regression structure, and the tail index $\alpha$ is taken to be constant across covariate values. Hence, directions for future research include allowing the tail index to depend on covariates. It would also be interesting to extend the compactification approach to the full generalized extreme-value and generalized Pareto families, where the support depends on the parameter. Further work is needed to treat dependent or nonidentically distributed observations. Finally, extending the theory to the domain-of-attraction framework would elevate the practical relevance of our asymptotic results.

\section*{Declaration of the use of generative AI}

During the preparation of this manuscript, the authors used Claude Code and ChatGPT for support in checking mathematical proofs, tidying up \LaTeX\ code, ensuring notational consistency, and concise language editing. All AI-assisted output was critically reviewed and verified by the authors, who take full responsibility for the manuscript.

\bibliographystyle{abbrvnat}
\bibliography{References}

\appendix
\crefalias{section}{appendix}
\section{\texorpdfstring{Proofs for \Cref{sec:framework}}{Proofs for Section 2, Fréchet Regression Model and Conditional Maximum Likelihood Estimator}}
\label{sec:proofs_framework}

First we prove an auxiliary result on the profile in $\alpha$ when maximizing the criterion over $\bbeta$ alone.
Fix $k \in \naturals$. For $(\by, \oX) \in \Rpos^k \times \reals^{k \times p}$, put
\begin{equation}
    \label{eq:profile}
    S(\alpha \mid \by, \oX) := \sup_{\bbeta \in \Rp} \tm[k]\rbr{(\alpha, \bbeta) \mid \by, \oX} ,
    \qquad \alpha \in \Rpos ,
\end{equation}
and write
\begin{equation}
    \label{eq:distance}
    \operatorname{d}(\oX, \by) := \inf_{\bb \in \Rp} \norm{\oX \bb - \y}_1
\end{equation}
for the $\ell_1$ distance from $\y$ to the column space of $\oX$; since that space is closed, $\operatorname{d}(\oX, \by) > 0$ if and only if $\y \notin \operatorname{col}(\oX)$.

\begin{lemma}[The profile in $\alpha$]
    \label{lem:profile}
    Let $k \in \naturals$, let $(\by, \oX) \in \Rpos^k \times \reals^{k \times p}$, and let $S$ and $\operatorname{d}$ be as in \Cref{eq:profile,eq:distance}. Then
    \begin{enumerate}[label=(\roman*)]
        \item $S(\cdot \mid \by, \oX)$ is finite and concave on $\Rpos$, and therefore continuous;
        \item $S(\alpha \mid \by, \oX) \leq \ln \alpha - 1$ for every $\alpha \in \Rpos$, so that $S(\alpha \mid \by, \oX) \to -\infty$ as $\alpha \downarrow 0$;
        \item $S(\alpha \mid \by, \oX) \leq \ln \alpha - \alpha \operatorname{d}(\oX, \by) / k$ for every $\alpha \in \Rpos$. If $\y \notin \operatorname{col}(\oX)$, so that $\operatorname{d}(\oX, \by) > 0$, then in addition $S(\alpha \mid \by, \oX) \to -\infty$ as $\alpha \to \infty$.
    \end{enumerate}
\end{lemma}

\begin{lproof}[Proof of \Cref{lem:profile}]
    \Cref{eq:psi} at sample size $k$ reads
    \[
        \psi(\alpha, \bb \mid \by, \oX) = \ln \alpha + \frac1k \sum_{i=1}^k g\rbr{\bb^\tr \bx_i - \alpha \ln y_i},
        \qquad (\alpha, \bb) \in \Rpos \times \Rp .
    \]
    For fixed $\alpha \in \Rpos$, the map $\bbeta \mapsto \bb = \alpha \bbeta$
    is a bijection of $\Rp$ onto itself, so that the supremum defining $S$ in \Cref{eq:profile} may be taken over $\bb$ instead of over $\bbeta$:
    \begin{equation}
        \label{eq:profile_as_partial_sup}
        S(\alpha \mid \by, \oX) = \sup_{\bb \in \Rp} \psi(\alpha, \bb \mid \by, \oX) ,
        \qquad \alpha \in \Rpos .
    \end{equation}

    \emph{Part~(i).} Collecting the $k$ arguments of $g$ into a vector,
    \begin{equation}
        \label{eq:affine_arguments}
        \rbr{\bb^\tr \bx_i - \alpha \ln y_i}_{i=1}^k = \oX \bb - \alpha \y ,
    \end{equation}
    so that $\psi(\cdot, \cdot \mid \by, \oX)$ is the sum of $\ln \alpha$ and of $k^{-1} \sum_{i=1}^k g$ evaluated at the coordinates of the linear map $(\alpha, \bb) \mapsto \oX \bb - \alpha \y$. Each summand is thus a concave function composed with an affine map, since $g''(t) = -\exp(t) < 0$, and $\ln$ is concave on $\Rpos$; hence $\psi(\cdot, \cdot \mid \by, \oX)$ is concave on the convex set $\Rpos \times \Rp$. Partial maximization preserves concavity: for $\alpha_1, \alpha_2 \in \Rpos$, $t \in [0,1]$ and arbitrary $\bb_1, \bb_2 \in \Rp$, the point $\rbr{t\alpha_1 + (1-t)\alpha_2, \; t\bb_1 + (1-t)\bb_2}$ lies in $\Rpos \times \Rp$, so that
    \[
        \psi\rbr{t\alpha_1 + (1-t)\alpha_2, \; t\bb_1 + (1-t)\bb_2 \mid \by, \oX}
        \geq t \, \psi(\alpha_1, \bb_1 \mid \by, \oX) + (1-t) \, \psi(\alpha_2, \bb_2 \mid \by, \oX) ,
    \]
    and taking the supremum over $\bb_1$ and over $\bb_2$ separately yields, through \Cref{eq:profile_as_partial_sup}, the concavity of $S(\cdot \mid \by, \oX)$.

    For the finiteness, the upper bound is part~(ii) and the lower bound follows from
    \[
        S(\alpha \mid \by, \oX) \geq \psi(\alpha, \0 \mid \by, \oX)
        = \ln \alpha + \frac1k \sum_{i=1}^k g\rbr{-\alpha \ln y_i} > -\infty .
    \]
    A finite concave function on an open interval is continuous.

    \emph{Part~(ii).} By the first bound of \Cref{eq:g_bounds}, every summand of $\psi$ is at most $-1$, so that $\psi(\alpha, \bb \mid \by, \oX) \leq \ln \alpha - 1$ for every $\bb \in \Rp$.
    As $\ln \alpha \to -\infty$ for $\alpha \downarrow 0$, the limit follows.

    \emph{Part~(iii).} Apply the second bound of \Cref{eq:g_bounds} to each of the $k$ summands, with $t = \bb^\tr \bx_i - \alpha \ln y_i$. Summing the resulting inequalities yields
    \begin{equation}
        \label{eq:bound1}
        \psi(\alpha, \bb \mid \by, \oX) \leq \ln \alpha - \frac1k \norm{\oX \bb - \alpha \y}_1 ,
        \qquad (\alpha, \bb) \in \Rpos \times \Rp .
    \end{equation}
    Since $\norm{\oX \bb - \alpha \y}_1 \geq \alpha \operatorname{d}(\oX, \by)$,
    taking the supremum over $\bb$ yields the bound of part~(iii). If $\y \notin \operatorname{col}(\oX)$, then $\operatorname{d}(\oX, \by) > 0$,
    and the right-hand side of that bound tends to $-\infty$ as $\alpha \to \infty$.
\end{lproof}

\begin{pproof}[Proof of \Cref{prop:existence_CMLE}\,(i)]
    The concavity of $\psi(\cdot, \cdot \mid \by, \oX)$
    was established in the proof of \Cref{lem:profile}\,(i) for $k = n$.
    For the strictness we pass to the column space
    \[
        V := \operatorname{col}(\oX) \subseteq \reals^n ,
    \]
    and we define
    \begin{equation}
        \label{eq:varphi}
        \varphi(\alpha, \bw) := \ln \alpha + \frac1n \sum_{i=1}^n g\rbr{w_i - \alpha \ln y_i},
        \qquad (\alpha, \bw) \in \Rpos \times V ,
    \end{equation}
    so that $\psi(\alpha, \bb \mid \by, \oX) = \varphi(\alpha, \oX \bb)$ by \Cref{eq:affine_arguments}. Since the functions $\ln$ and $g$ are strictly concave, the function $\varphi$ is easily seen to be strictly concave on $\Rpos \times V$.

    Assume now~(a). Then $\bb \mapsto \oX \bb$ is injective, so that $(\alpha, \bb) \mapsto (\alpha, \oX \bb)$ is an injective affine map of $\Rpos \times \Rp$ onto $\Rpos \times V$. A strictly concave function composed with an injective affine map is strictly concave; hence so is $\psi(\cdot, \cdot \mid \by, \oX)$.
\end{pproof}

\begin{pproof}[Proof of \Cref{prop:existence_CMLE}\,(ii)]
    Assume~(b). By \Cref{eq:psi}, the function $\psi(\cdot, \cdot \mid \by, \oX)$ is $\tm[n](\cdot \mid \by, \oX)$ in the coordinates $(\alpha, \bb) = (\alpha, \alpha \bbeta)$, and $(\alpha, \bbeta) \mapsto (\alpha, \alpha \bbeta)$ is a bijection of $\Rpos \times \Rp$ onto itself. It therefore suffices to show that $\psi(\cdot, \cdot \mid \by, \oX)$ attains its supremum on $\Rpos \times \Rp$. Let $V$ and the function $\varphi$ of \Cref{eq:varphi} be as in the proof of part~(i), so that $\psi(\alpha, \bb \mid \by, \oX) = \varphi(\alpha, \oX \bb)$. As $\bb \mapsto \oX \bb$ maps $\Rp$ onto $V$, the two functions have the same range and the same suprema, and it is enough to maximize $\varphi$. For the same reason, \Cref{eq:profile_as_partial_sup} at $k = n$ implies
    \begin{equation}
        \label{eq:varphi_profile}
        \sup_{\bw \in V} \varphi(\alpha, \bw) = S(\alpha \mid \by, \oX) ,
        \qquad \alpha \in \Rpos.
    \end{equation}

    \emph{Step 1: the two coercivity bounds.} Write $d := \operatorname{d}(\oX, \by)$ for the distance of \Cref{eq:distance}, which is strictly positive by hypothesis~(b). By \Cref{eq:bound1} at $\bb$ with $\bw = \oX \bb$, and by \Cref{eq:varphi_profile} together with parts~(ii) and~(iii) of \Cref{lem:profile}, we obtain, for all $\alpha > 0$ and $\bw \in V$,
    \begin{equation}
        \label{eq:coercive}
        \varphi(\alpha, \bw)
        \leq \ln \alpha - \frac1n \norm{\bw - \alpha \y}_1
        \leq \ln \alpha - \frac{\alpha d}{n}
        \mtext{and}
        \varphi(\alpha, \bw) \leq \ln \alpha - 1 .
    \end{equation}

    \emph{Step 2: the superlevel sets of $\varphi$ are compact.} Fix $L \in \reals$ and put
    \[
        T_L := \cbr{(\alpha, \bw) \in \Rpos \times V : \varphi(\alpha, \bw) \geq L} .
    \]
    Let $(\alpha, \bw) \in T_L$. Three bounds follow in turn.
    \begin{itemize}[leftmargin=*]
        \item
        The second bound in \Cref{eq:coercive} gives $L \leq \ln \alpha - 1$, that is, $\alpha \geq a := \exp(L + 1) > 0$.
        \item
        The first bound in \Cref{eq:coercive} gives $L \leq \ln \alpha - \alpha d / n$. Since $d > 0$, the right-hand side tends to $-\infty$ as $\alpha \to \infty$, so the set of $\alpha$ satisfying this inequality is bounded above by some finite $A = A(L)$.
        \item
        Using the first inequality of \Cref{eq:coercive} once more, now keeping the norm, gives
        \[
            \frac1n \norm{\bw - \alpha \y}_1 \leq \ln \alpha - L \leq \ln A - L ,
        \]
        and therefore, by the triangle inequality and $\alpha \leq A$,
        \[
            \norm{\bw}_1
            \leq \norm{\bw - \alpha \y}_1 + \alpha \norm{\y}_1
            \leq n \rbr{\ln A - L} + A \norm{\y}_1 =: R .
        \]
    \end{itemize}
    Hence $T_L$ is contained in the compact set $[a, A] \times \cbr{\bw \in V : \norm{\bw}_1 \leq R}$.
    As $\varphi$ is continuous, $T_L$ is closed and thus compact too.

    \emph{Step 3: $\varphi$ has exactly one maximizer.} Take $L_0 := \varphi(1, \0)$, which is finite. The set $T_{L_0}$ is nonempty, since it contains $(1, \0)$, and compact, so $\varphi$ attains a maximum on it at some $(\hat\alpha, \hat\bw)$. Moreover, we have $\varphi < L_0 \leq \varphi(\hat\alpha, \hat\bw)$ outside $T_{L_0}$. It follows that $(\hat\alpha, \hat\bw)$ maximizes $\varphi$ over $\Rpos \times V$. That maximizer is unique, $\varphi$ being strictly concave on the convex set $\Rpos \times V$ by the proof of part~(i).

    \emph{Step 4: back to $\Eta$.} Choose any $\hat\bb \in \Rp$ with $\oX \hat\bb = \hat\bw$. Then $\psi(\hat\alpha, \hat\bb \mid \by, \oX) = \varphi(\hat\alpha, \hat\bw)$ is the maximum of $\psi(\cdot, \cdot \mid \by, \oX)$ over $\Rpos \times \Rp$, so that $(\hat\alpha, \hat\bbeta)$ with $\hat\bbeta := \hat\bb / \hat\alpha$ maximizes $\tm[n](\cdot \mid \by, \oX)$ over $\Eta$. For the set of all maximizers, recall that $\psi(\alpha, \bb \mid \by, \oX) = \varphi(\alpha, \oX \bb)$. By Step~3, a pair $(\alpha, \bb)$ maximizes $\psi(\cdot, \cdot \mid \by, \oX)$ if and only if $(\alpha, \oX \bb) = (\hat\alpha, \hat\bw)$, that is, if and only if $\alpha = \hat\alpha$ and $\bb \in \hat\bb + \ker(\oX)$. Dividing the second coordinate by $\hat\alpha$ and using that $\ker(\oX)$ is a subspace,
    we find that the set of maximizers of $\tm[n](\cdot \mid \by, \oX)$ is $\cbr{\hat\alpha} \times (\hat\bbeta + \ker(\oX))$.
\end{pproof}

\begin{pproof}[Proof of \Cref{prop:existence_CMLE}\,(iii)]
    Under~(a), part~(i) makes $\psi(\cdot, \cdot \mid \by, \oX)$ strictly concave, so it has at most one maximizer; under~(b), part~(ii) gives at least one. Since $(\alpha, \bbeta) \mapsto (\alpha, \alpha \bbeta)$ is a bijection of $\Rpos \times \Rp$ onto itself, $\tm[n](\cdot \mid \by, \oX)$ has exactly one maximizer.
\end{pproof}

\begin{pproof}[Proof of \Cref{prop:identifiability}]
    Write $\sigma_a(\bx) := \exp(\bbeta_a^\tr \bx)$ and $\sigma_b(\bx) := \exp(\bbeta_b^\tr \bx)$ for the two scale functions.

    Suppose first that~(ii) holds. Then $\sigma_a(\bX) = \exp(\bbeta_a^\tr \bX) = \exp(\bbeta_b^\tr \bX) = \sigma_b(\bX)$ almost surely, and since by \Cref{eq:pdf} the conditional density depends on $\bmeta$ and $\bx$ only through the pair $(\alpha, \sigma(\bx))$, the two densities coincide on the same event, yielding~(i).

    Conversely, suppose that~(i) holds and fix $\bx$ in the almost sure event on which the two densities agree Lebesgue-almost everywhere,
    so
    \[
        \Frechet\rbr{\alpha_a, \sigma_a(\bx)} = \Frechet\rbr{\alpha_b, \sigma_b(\bx)} .
    \]
    Since the two parameters of the Fréchet family are identifiable, we find
    $\alpha_a = \alpha_b$ and $\sigma_a(\bx) = \sigma_b(\bx)$, the latter being $\rbr{\bbeta_a - \bbeta_b}^\tr \bx = 0$ after taking logarithms. As $\bx$ ranged over an almost sure event, this is~(ii).

    For the final statement, write $\ba := \bbeta_a - \bbeta_b$. If \Cref{cond:identifiability} holds and~(i) is satisfied, then~(ii) gives $\ba^\tr \bX = 0$ almost surely, which forces $\ba = \0$ and hence $\bmeta_a = \bmeta_b$. If \Cref{cond:identifiability} fails, pick $\ba \neq \0$ with $\pr[\ba^\tr \bX = 0] = 1$ and set $\bmeta_b := (\alpha_a, \bbeta_a - \ba)$. Then~(ii) holds while $\bmeta_b \neq \bmeta_a$, and the parameter is not identifiable.
\end{pproof}

\begin{lproof}[Proof of \Cref{lem:unique_maximizer}]
    We show that $\bmeta_\dagger$ is the unique maximizer of the asymptotic criterion function $M$ of \Cref{eq:asymptotic_criterion_function}, which by \Cref{eq:m_upper_bound} is well defined in $[-\infty, \infty)$ for every $\bmeta \in \Eta$. Since $f_{\bmeta_\dagger}^{Y\mid \bX}(\,\cdot \mid \bx)$ is the Lebesgue density of $P_{\bmeta_\dagger}^{(Y\mid\bX)}(\,\cdot \mid \bx)$ and is strictly positive on $\Rpos$, conditioning on $\bX$ and applying Jensen's inequality to the concave function $\ln$ yields
    \begin{align*}
        M(\bmeta) - M(\bmeta_\dagger)
        = \ex_{\bmeta_\dagger}\sbr{\ln \rbr{\frac{f_\bmeta^{Y\mid \bX}}{f_{\bmeta_\dagger}^{Y\mid \bX}}}}
        & = \int_{\Rp} \int_{\Rpos} \ln \rbr{\frac{f_\bmeta^{Y\mid \bX}(y\mid\bx)}{f_{\bmeta_\dagger}^{Y\mid \bX}(y\mid\bx)}} f_{\bmeta_\dagger}^{Y\mid \bX}(y\mid\bx) \, dy \, dP_{\bX}(\bx) \\
        & \leq \int_{\Rp} \ln \rbr{\int_{\Rpos} \frac{f_\bmeta^{Y\mid \bX}(y\mid\bx)}{f_{\bmeta_\dagger}^{Y\mid \bX}(y\mid\bx)} f_{\bmeta_\dagger}^{Y\mid \bX}(y\mid\bx) \, dy} \, dP_{\bX}(\bx) \\
        & = \int_{\Rp} \ln \rbr{\int_{\Rpos} f_\bmeta^{Y\mid \bX}(y\mid\bx) \, dy} \, dP_{\bX}(\bx)
        = 0.
    \end{align*}
    The first equality of the display is \Cref{eq:M_difference}; recall that $M(\bmeta_\dagger)$ is finite by \Cref{eq:m_at_truth}. In particular $M(\bmeta) \leq M(\bmeta_\dagger)$ for every $\bmeta \in \Eta$, so $\bmeta_\dagger$ is a maximizer.

    It remains to show that the inequality is strict as soon as $\bmeta \neq \bmeta_\dagger$. Write
    \[
        I(\bx) := \int_{\Rpos} \ln \rbr{\frac{f_\bmeta^{Y\mid \bX}(y\mid\bx)}{f_{\bmeta_\dagger}^{Y\mid \bX}(y\mid\bx)}} f_{\bmeta_\dagger}^{Y\mid \bX}(y\mid\bx) \, dy
    \]
    for the inner integral in the first line of the display, so that $M(\bmeta) - M(\bmeta_\dagger) = \int_{\Rp} I(\bx) \, dP_{\bX}(\bx)$, and note that the second and third lines of the display show that $I(\bx) \leq \ln 1 = 0$ for every $\bx$. The logarithm is \emph{strictly} concave, so Jensen's inequality applied conditionally on $\bX = \bx$ is an equality if and only if the likelihood ratio $f_\bmeta^{Y\mid \bX}(Y\mid\bx) / f_{\bmeta_\dagger}^{Y\mid \bX}(Y\mid\bx)$ is almost surely constant under $P_{\bmeta_\dagger}^{(Y\mid\bX)}(\,\cdot \mid \bx)$.
    Such a constant is equal to the expectation of the ratio, which then must be equal to one, that is, the two densities agree Lebesgue-almost everywhere.
    For $\bmeta \neq \bmeta_\dagger$, \Cref{prop:identifiability} says exactly that, under \Cref{cond:identifiability}, there is a positive probability that this equality of densities does not happen.
    Therefore $I \leq 0$ everywhere and $I < 0$ on a set of positive $P_{\bX}$-probability, whence
    $M(\bmeta) < M(\bmeta_\dagger)$ for every $\bmeta \in \Eta \setminus \cbr{\bmeta_\dagger}$.
    We conclude that $\Eta_0 = \cbr{\bmeta_\dagger}$.
\end{lproof}

\begin{lproof}[Proof of \Cref{lem:design_regularity}]
    \emph{Part~(i).} Since $\rank(\mX) \leq p < k$, the column space $\operatorname{col}(\oX)$ is a proper linear subspace of $\reals^k$ for every realization $\oX$, whence $\leb_k\rbr{\operatorname{col}(\oX)} = 0$, with $\leb_k$ the Lebesgue measure on $\reals^k$. By \Cref{eq:model}, conditionally on $\mX = \oX$ the vector $\Y$ is distributed as $\oX \bd + \ad^{-1} \bG$, with $\bG$ a vector of $k$ iid standard Gumbel variables, and is therefore absolutely continuous with respect to $\leb_k$. Consequently,
    \[
        \pr_{\bmeta_\dagger}\sbr{\Y \in \operatorname{col}(\mX) \mid \mX = \oX} = 0
    \]
    for $P_{\mX}$-almost every $\oX$, and taking expectations gives $\pr_{\bmeta_\dagger}\sbr{\Y \in \operatorname{col}(\mX)} = 0$.

    \emph{Part~(ii).} Let $S \subseteq \Rp$ be the support of the distribution of $\bX$, that is, the smallest closed set with probability one. \Cref{cond:identifiability} implies that the linear span of $S$ is equal to $\Rp$.

    Pick $\bx_1, \ldots, \bx_p \in S$ that are linearly independent. For $\bz_1,\ldots,\bz_p \in \Rp$, let $D(\bz_1, \ldots, \bz_p)$ denote the determinant of the $p \times p$ matrix with rows $\bz_1^\tr, \ldots, \bz_p^\tr$. Then $D(\bx_1, \ldots, \bx_p) \neq 0$, and $D$ is continuous, so there is $\varepsilon > 0$ such that $D(\bz_1, \ldots, \bz_p) \neq 0$ as soon as $\norm{\bz_j - \bx_j} < \varepsilon$ for every~$j$. Shrinking $\varepsilon$ if necessary, the open balls $B_j := \cbr{\bz \in \Rp : \norm{\bz - \bx_j} < \varepsilon}$ may in addition be taken pairwise disjoint. Each $\bx_j$ lies in the support, so that
    \[
        q := \min_{j = 1, \ldots, p} \pr\sbr{\bX \in B_j} > 0 .
    \]

    Write $A_k := \cbr{\rank(\mX) = p}$, the $k \times p$ matrix $\mX$ being built from the first $k$ observations. If each of the $p$ disjoint balls contains at least one of $\bX_1, \ldots, \bX_k$, then
    $A_k$ occurs. By independence and a union bound,
    \[
        1 - \pr\sbr{A_k}
        \leq \pr\sbr{\bigcup_{j=1}^p \cbr{\bX_1 \notin B_j, \ldots, \bX_k \notin B_j}}
        \leq p \rbr{1-q}^k
        \to 0,
    \]
    as $k\to\infty$.
    Adding rows cannot decrease the rank, so the events $A_k$ increase with $k$ and $\pr\sbr{\bigcup_{k \geq p} A_k} = \lim_{k \to \infty} \pr\sbr{A_k} = 1$. On that union, $\rank(\mX) = p$ for all $k$ large enough.
\end{lproof}

\section{\texorpdfstring{Proofs for \Cref{sec:wald}}{Proofs for Section 3, An Extension of Wald's Consistency Theorem}}
\label{sec:proofs_wald}

\begin{lemma}[Local suprema]
    \label{lem:local_suprema}
    Let $(\Theta, d)$ be a metric space and let $W$, $\cW$, $P = P_W$ and $\varphi$ be as in \Cref{def:upper_semicontinuity}. For a nonempty set $U \subseteq \Theta$, write
    \[
        \varphi_U(w) := \sup \cbr{\varphi_{\bar\theta}(w) : \bar\theta \in U},
        \qquad w \in \cW,
    \]
    and recall that $B(\theta, \rho) = \cbr{\bar\theta \in \Theta : d(\bar\theta, \theta) < \rho}$ is the open ball of radius $\rho > 0$ around $\theta$. Fix $\theta \in \Theta$ and assume that the random map $\theta' \mapsto \varphi_{\theta'}(W)$ is almost surely upper-semicontinuous at $\theta$ (\Cref{def:upper_semicontinuity}) and satisfies Wald's integrability condition at $\theta$ (\Cref{cond:wald_integrability}), the latter with a radius $\delta > 0$. Then the expectations $\ex\sbr{\varphi_\theta(W)}$ and $\ex[\varphi_{B(\theta, \rho)}(W)]$, for $\rho \in (0, \delta]$, are well defined in $[-\infty, \infty)$, and
    \begin{equation}
        \label{eq:local_suprema}
        \ex\sbr{\varphi_{B(\theta, \rho)}(W)} \downarrow \ex\sbr{\varphi_\theta(W)}, \mtext{as} \rho \downarrow 0.
    \end{equation}
\end{lemma}
\begin{lproof}[Proof of \Cref{lem:local_suprema}]
    Let $\delta$ be a radius as in \Cref{cond:wald_integrability} at $\theta$, let $\cN_\theta$ be a null set as in \Cref{def:upper_semicontinuity} at $\theta$, and abbreviate $U := B(\theta, \delta)$ and $U_t := B(\theta, \delta / t)$ for $t \in \naturals$. Since $\theta \in U_{t+1} \subseteq U_t \subseteq U$ for every $t$,
    we have, pointwise on $\cW$,
    \begin{equation}
        \label{eq:local_suprema_sandwich}
        \varphi_\theta(w) \leq \varphi_{U_{t+1}}(w) \leq \varphi_{U_t}(w) \leq \varphi_U(w),
        \qquad t \in \naturals.
    \end{equation}
    Taking positive parts preserves these inequalities, so item~(b) of \Cref{cond:wald_integrability} makes the positive parts of $\varphi_\theta(W)$ and of $\varphi_{B(\theta, \rho)}(W)$, for $\rho \in (0, \delta]$, integrable; all expectations in \Cref{eq:local_suprema} are therefore well defined in $[-\infty, \infty)$. Moreover, $\rho \mapsto \ex[\varphi_{B(\theta, \rho)}(W)]$ is nondecreasing, so its limit for $\rho \downarrow 0$ exists and is equal to $\lim_{t \to \infty} \ex[\varphi_{U_t}(W)]$. It is thus enough to show that the latter limit is $\ex\sbr{\varphi_\theta(W)}$.

    \emph{Step 1: pointwise convergence.} We claim that
    \begin{equation}
        \label{eq:local_suprema_pointwise}
        \lim_{t \to \infty} \varphi_{U_t}(w) = \varphi_\theta(w),
        \qquad w \in \cW \setminus \cN_\theta.
    \end{equation}
    Fix $w \in \cW \setminus \cN_\theta$. By \Cref{eq:local_suprema_sandwich}, the sequence $\rbr{\varphi_{U_t}(w)}_{t \in \naturals}$ is nonincreasing and bounded from below by $\varphi_\theta(w)$, so that its limit exists in $[-\infty, \infty]$ and is at least $\varphi_\theta(w)$. If that limit is $-\infty$, then the same bound forces $\varphi_\theta(w) = -\infty$ as well, and \Cref{eq:local_suprema_pointwise} holds at $w$. Assume therefore that the limit is larger than $-\infty$, and let $c \in \reals$ be such that $c < \lim_{t \to \infty} \varphi_{U_t}(w)$. Then $\varphi_{U_t}(w) > c$ for every $t$, so that, by the definition of a supremum, there exists $\theta_t \in U_t$ with $\varphi_{\theta_t}(w) > c$. Since $d(\theta_t, \theta) < \delta / t$, the sequence $\rbr{\theta_t}_{t \in \naturals}$ converges to $\theta$, and the upper-semicontinuity at $\theta$ yields
    \[
        \varphi_\theta(w) \geq \limsup_{t \to \infty} \varphi_{\theta_t}(w) \geq c.
    \]
    As $c$ was an arbitrary real number less than $\lim_{t \to \infty} \varphi_{U_t}(w)$, this proves \Cref{eq:local_suprema_pointwise}.

    \emph{Step 2: monotone convergence.} Suppose first that $\ex\sbr{\varphi_U(W)} = -\infty$. By \Cref{eq:local_suprema_sandwich}, all the expectations in \Cref{eq:local_suprema} are then equal to $-\infty$ as well, and there is nothing left to prove. Assume therefore that $\ex\sbr{\varphi_U(W)} > -\infty$. Together with item~(b) of \Cref{cond:wald_integrability}, this makes $\varphi_U(W)$ integrable. Define $h_t : \cW \to [0, \infty]$, for $t \in \naturals$, by
    \[
        h_t(w) :=
        \begin{dcases}
            \varphi_U(w) - \varphi_{U_t}(w) & \text{if } \abr{\varphi_U(w)} < \infty, \\
            0 & \text{otherwise,}
        \end{dcases}
    \]
    and let $h_\infty : \cW \to [0, \infty]$ be defined in the same way, with $\varphi_{U_t}$ replaced by $\varphi_\theta$.
    By \Cref{eq:local_suprema_sandwich}, these functions are nonnegative and $h_t$ is nondecreasing in $t$. On $\cbr{\abr{\varphi_U} = \infty}$ they all vanish, so that $h_t \uparrow h_\infty$ there trivially, while on the complement this is \Cref{eq:local_suprema_pointwise}; the convergence thus holds on all of $\cW \setminus \cN_\theta$. The monotone convergence theorem therefore gives $\ex\sbr{h_t(W)} \to \ex\sbr{h_\infty(W)}$ in $[0, \infty]$ as $t \to \infty$. Finally, the integrability of $\varphi_U(W)$ makes $\abr{\varphi_U} < \infty$ hold $P$-almost surely, so that $h_t = \varphi_U - \varphi_{U_t}$ and $h_\infty = \varphi_U - \varphi_\theta$ almost surely; as $-\varphi_{U_t}(W)$ and $-\varphi_\theta(W)$ have integrable negative parts, the expectations may be split,
    \[
        \ex\sbr{h_t(W)} = \ex\sbr{\varphi_U(W)} - \ex\sbr{\varphi_{U_t}(W)},
        \qquad
        \ex\sbr{h_\infty(W)} = \ex\sbr{\varphi_U(W)} - \ex\sbr{\varphi_\theta(W)},
    \]
    both in $(-\infty, \infty]$. Canceling the finite term $\ex\sbr{\varphi_U(W)}$ yields $\ex[\varphi_{U_t}(W)] \downarrow \ex\sbr{\varphi_\theta(W)}$, as required.
\end{lproof}

\begin{lemma}
    \label{lem:summand_integrability}
    Let $l \in \naturals$ and let $V_1, \ldots, V_l$ be independent and identically distributed random variables with values in $[-\infty, \infty)$. If $\ex\sbr{\rbr{\sum_{i=1}^l V_i}_+} < \infty$, then $\ex\sbr{\rbr{V_1}_+} < \infty$.
\end{lemma}

\begin{lproof}[Proof of \Cref{lem:summand_integrability}]
    For $l = 1$ there is nothing to prove, so let $l \geq 2$. Suppose $\ex\sbr{\rbr{V_1}_+} = \infty$; we show that $\ex\sbr{\rbr{\sum_{i=1}^l V_i}_+} = \infty$ as well. Then $q := \pr\sbr{V_1 \geq 0} > 0$, since otherwise $\rbr{V_1}_+$ would vanish almost surely. On the event $A := \cbr{\min(V_2, \ldots, V_l) \geq 0}$, the summands $V_2, \ldots, V_l$ are finite and nonnegative, so that $\sum_{i=1}^l V_i \geq V_1$ and hence $\rbr{\sum_{i=1}^l V_i}_+ \geq \rbr{V_1}_+$. Since $V_1$ is independent of $(V_2, \ldots, V_l)$ and the variables are identically distributed, $\pr\sbr{A} = q^{l-1}$ and
    \begin{align*}
        \ex\sbr{\rbr{\sum_{i=1}^l V_i}_+}
        &\geq \ex\sbr{\rbr{\sum_{i=1}^l V_i}_+ \mathds{1}_{A}} \\
        &\geq \ex\sbr{\rbr{V_1}_+} \pr\sbr{A} = \ex\sbr{\rbr{V_1}_+} q^{l-1} = \infty .
        \qedhere
    \end{align*}
\end{lproof}

\paragraph{Remark.} Applied with $V_i = \mm{\theta}{Z_i}$, \Cref{lem:summand_integrability} shows that \Cref{cond:wald_integrability} does constrain the criterion of a \emph{single} observation, although in \Cref{thm:wald_blocks} it is imposed on a block of $l$ of them. Indeed, for $\theta \in \sTheta$, taking $\theta' = \theta$ in the supremum of item~(b) and using assumption~(i) of that theorem gives $\ex\sbr{\rbr{\sum_{i=1}^l \m{Z_i}}_+} < \infty$, the factor $1/l$ being irrelevant to integrability, and therefore
\begin{equation}
    \label{eq:single_observation_integrability}
    \ex\sbr{\rbr{\m{Z}}_+} < \infty, \qquad \theta \in \sTheta .
\end{equation}
The expectation $\ex\sbr{\m{Z}}$ is thus well defined in $[-\infty, \infty)$, and, no expression $\infty - \infty$ being able to arise, $\ex\sbr{\om\rbr{\theta \mid \bZ^{(l)}}} = \ex\sbr{\m{Z}}$ for every $\theta \in \sTheta$, whatever the block size $l$.

\begin{proof}[Proof of \Cref{thm:wald_blocks}]
    Fix $k \in \naturals$ and let $l \in \cbr{k, \ldots, 2k-1}$. Blocks of $l$ observations take their values in $\cZ^l$, and we write
    \begin{equation}
        \label{eq:extended_block_objective}
        \varphi^{(l)}_\theta\rbr{\bz^{(l)}} := \om\rbr{\theta \mid \bz^{(l)}},
        \qquad
        \varphi^{(l)}_U\rbr{\bz^{(l)}} := \sup_{\theta' \in U} \varphi^{(l)}_{\theta'}\rbr{\bz^{(l)}},
    \end{equation}
    for $\bz^{(l)} = \rbr{z_1, \ldots, z_l} \in \cZ^l$, for $\theta \in \Theta$ and for a nonempty subset $U \subseteq \Theta$. The generic block $\bZ^{(l)}$ of \Cref{sec:wald} has distribution $\Pl$, the $l$-fold product measure of $P_Z$, and, by assumption~(i), $\varphi^{(l)}_\theta\rbr{\bZ^{(l)}} = l^{-1} \sum_{i=1}^l \mm{\theta}{Z_i}$ for every $\theta \in \sTheta$.

    If the supremum in \Cref{eq:Theta_0_blocks} equals $-\infty$, then
    the assertion of the theorem holds trivially. Assume from now on that the supremum is larger than $-\infty$.

    \emph{Step 1: finiteness of the maximum, and the local suprema.} Let $\theta_0$ be the point of $\mTheta$ appearing in the near-maximizer hypothesis of \Cref{thm:wald_blocks}. Its expectation $\ex\sbr{\mm{\theta_0}{Z}}$ is the supremum in \Cref{eq:Theta_0_blocks} and is thus larger than $-\infty$. In view of \Cref{eq:single_observation_integrability}, we find that $\mm{\theta_0}{Z}$ is integrable; write
    \begin{equation}
        \label{eq:maximum_finite}
        \mu_0 := \ex\sbr{\mm{\theta_0}{Z}} \in \reals.
    \end{equation}

    We now apply \Cref{lem:local_suprema} with $\cW = \cZ^l$, with $W = \bZ^{(l)}$ and $P = \Pl$, and with $\varphi^{(l)}$ as in \Cref{eq:extended_block_objective}. Its hypotheses hold at every $\theta \in \Theta$:
    \begin{itemize}
        \item $\varphi^{(l)}_\theta$ is measurable on $\cZ^l$, this being assumed of $\om$ in \Cref{thm:wald_blocks};
        \item the almost sure upper-semicontinuity at $\theta$ that the lemma requires is assumed in condition~(ii) of \Cref{thm:wald_blocks}. Indeed, with $\cW = \cZ^l$ and $W = \bZ^{(l)}$, \Cref{def:upper_semicontinuity} supplies a set $\cN_\theta^{(l)} \subseteq \cZ^l$ with $\Pl\sbr{\cN_\theta^{(l)}} = 0$ such that $\limsup_{j \to \infty} \varphi^{(l)}_{\theta_j}\rbr{\bz^{(l)}} \leq \varphi^{(l)}_\theta\rbr{\bz^{(l)}}$ for every $\bz^{(l)} \notin \cN_\theta^{(l)}$ and every sequence $\theta_j \to \theta$; the lemma allows its null set to depend on $\theta$ in the same way;
        \item \Cref{cond:wald_integrability} at $\theta$, the second hypothesis of the lemma, is assumed in the same condition~(ii), again with $\cW = \cZ^l$ and $W = \bZ^{(l)}$.
    \end{itemize}
    For every $\theta \in \Theta$, the lemma therefore yields
    \begin{equation}
        \label{eq:MCT_general}
        \ex\sbr{\sup_{\theta' \in B(\theta, \rho)} \varphi^{(l)}_{\theta'}\rbr{\bZ^{(l)}}}
        \downarrow
        \ex\sbr{\varphi^{(l)}_{\theta}\rbr{\bZ^{(l)}}},
        \mtext{as} \rho \downarrow 0,
    \end{equation}
    with both sides well defined in $[-\infty, \infty)$; for $\theta \in \sTheta$, the right-hand side is $\ex\sbr{\m Z}$.

    \emph{Step 2: the set $B$ and a finite cover of it.} Fix $\eps > 0$ and a compact set $K \subseteq \Theta$, and put
    \[
        B := \cbr{\theta \in K : d\rbr{\theta, \mTheta} \geq \eps} .
    \]
    Since $\theta \mapsto d\rbr{\theta, \mTheta}$ is continuous, $B$ is a closed subset of $K$ and therefore compact, and the assertion of the theorem is that $\pr[\est_n \in B] \to 0$. The estimator being $\sTheta$-valued, that event is the event $\cbr{\est_n \in B \cap \sTheta}$; if $B \cap \sTheta$ is empty, its probability is zero for every $n$, so assume that $B \cap \sTheta$ is nonempty. As $\eps > 0$, no point of $B$ belongs to $\mTheta$.

    Let $\theta \in B$. Since $\theta \notin \mTheta$, we have
    \[
        \ex\sbr{\varphi^{(k)}_\theta\rbr{\bZ^{(k)}}}
        < \ex\sbr{\varphi^{(k)}_{\theta_0}\rbr{\bZ^{(k)}}}
        = \mu_0 ,
    \]
    and this whether $\theta$ lies in $\sTheta$, where the left-hand side equals $\ex\sbr{\m Z}$ and $\theta$ fails to attain the supremum in \Cref{eq:Theta_0_blocks}, or outside it, where the left-hand side is $-\infty$ by assumption~(iii). Let $\delta_\theta > 0$ be the radius supplied by \Cref{cond:wald_integrability} at $\theta$ for blocks of size $k$. By \Cref{eq:MCT_general}, applied with $l = k$ at the point $\theta$, there exists a radius $\rho_\theta \in (0, \delta_\theta]$ with
    \begin{equation}
        \label{eq:ball_below_max_general}
        \ex\sbr{\varphi^{(k)}_{B\rbr{\theta, \rho_\theta}}\rbr{\bZ^{(k)}}} < \mu_0 .
    \end{equation}
    The balls $B\rbr{\theta, \rho_\theta}$, for $\theta \in B$, form an open cover of the compact set $B$. Choose a finite subcover, indexed by $\theta_1, \ldots, \theta_q \in B$ with radii $\rho_1, \ldots, \rho_q$, and put
    \begin{equation}
        \label{eq:c_definition_general}
        c := \max_{a = 1, \ldots, q} \ex\sbr{\varphi^{(k)}_{B\rbr{\theta_a, \rho_a}}\rbr{\bZ^{(k)}}} \in [-\infty, \infty) .
    \end{equation}
    Being a maximum of finitely many quantities each of which is smaller than $\mu_0$ by \Cref{eq:ball_below_max_general}, it satisfies $c < \mu_0$.

    \emph{Step 3: the remainder block.} Let $r \in \cbr{k, \ldots, 2k-1}$ and write $\bZ_0^{(r)} := \rbr{Z_1, \ldots, Z_r}$ for the remainder block, carrying the index $0$ because it is placed before the others in Step~4. Repeat the covering argument of Step~2 for blocks of size $r$: at each $\theta \in B$, take the ball $B\rbr{\theta, \sigma_\theta}$ with $\sigma_\theta$ the radius supplied by \Cref{cond:wald_integrability} at $\theta$ for that block size, and extract a finite subcover indexed by $\vartheta_1, \ldots, \vartheta_s \in B$ with radii $\sigma_1, \ldots, \sigma_s$. We obtain the random variable
    \begin{equation}
        \label{eq:remainder_bound_general}
        C_r\rbr{\bZ_0^{(r)}} := \max_{b = 1, \ldots, s} \varphi^{(r)}_{B\rbr{\vartheta_b, \sigma_b}}\rbr{\bZ_0^{(r)}} .
    \end{equation}
    It is measurable and has integrable positive part, being a maximum of finitely many such variables; in particular $\rbr{C_r\rbr{\bZ_0^{(r)}}}_+$ is finite almost surely. Further, it dominates the remainder uniformly over $B$, since every $\theta \in B$ lies in one of the balls of the subcover.

    \emph{Step 4: the sample split into $k$ subsequences.} For $r \in \cbr{k, \ldots, 2k-1}$ and $\ell \in \naturals$, put $n_r(\ell) := k \ell + r$. Every integer $n \geq 2k$ is of this form for exactly one pair $(r, \ell)$, namely $\ell = \left\lfloor n / k \right\rfloor - 1$ and $r = n - k\ell$, so that the $k$ sequences $\rbr{n_r(\ell)}_{\ell \in \naturals}$, one for each $r$, partition $\cbr{2k, 2k+1, \ldots}$ between them. A sequence of numbers converges to zero as soon as each of finitely many subsequences covering it does; it is therefore enough to prove, separately for each of the $k$ values of $r$, that
    \begin{equation}
        \label{eq:subsequence_goal}
        \pr\sbr{\est_{n_r(\ell)} \in B} \to 0, \mtext{as} \ell \to \infty .
    \end{equation}
    Fix $r$ from here on and abbreviate $n := n_r(\ell)$. The sample $Z_1, \ldots, Z_n$ is split by putting the remainder first: the first $r$ observations form $\bZ_0^{(r)}$, and the remaining $k\ell$ form $\ell$ blocks of size $k$,
    \[
        \bZ_i^{(k)} := \rbr{Z_{r + (i-1)k + 1}, \ldots, Z_{r + ik}},
        \qquad i = 1, \ldots, \ell .
    \]
    Because $r$ is held fixed, the $i$\textsuperscript{th} block is one and the same random vector for every $\ell$, and so is $\bZ_0^{(r)}$. The whole subsequence is thus driven by a single infinite iid sequence $\bZ_1^{(k)}, \bZ_2^{(k)}, \ldots$ of blocks, each distributed as $\bZ^{(k)}$, together with the one random vector $\bZ_0^{(r)}$, and letting $\ell \to \infty$ is an ordinary limit along that sequence.

    \emph{Step 5: a measurable upper bound for the criterion over $B$.} Let $\theta \in B \cap \sTheta$ and pick $a \in \cbr{1, \ldots, q}$ with $\theta \in B\rbr{\theta_a, \rho_a}$. On $\sTheta$ the criterion of a block is its average, by assumption~(i). Splitting $M_n(\theta)$ according to the blocks and bounding the remainder by \Cref{eq:remainder_bound_general} and each block by the supremum over the ball, we find
    \begin{align*}
        M_n(\theta)
        &= \frac rn \, \varphi^{(r)}_\theta\rbr{\bZ_0^{(r)}} + \frac kn \sum_{i=1}^\ell \varphi^{(k)}_\theta\rbr{\bZ_i^{(k)}} \\
        &\leq \frac rn \rbr{C_r\rbr{\bZ_0^{(r)}}}_+ + \frac kn \sum_{i=1}^\ell \varphi^{(k)}_{B\rbr{\theta_a, \rho_a}}\rbr{\bZ_i^{(k)}}
        \leq T_{r, \ell} ,
    \end{align*}
    where the random variable $T_{r,\ell}$ is defined by
    \begin{equation}
        \label{eq:T_definition_general}
        T_{r, \ell} := \frac{r \rbr{C_r\rbr{\bZ_0^{(r)}}}_+}{n_r(\ell)}
        + \max_{a = 1, \ldots, q} \frac{k}{n_r(\ell)} \sum_{i=1}^\ell \varphi^{(k)}_{B\rbr{\theta_a, \rho_a}}\rbr{\bZ_i^{(k)}} .
    \end{equation}
    The right-hand side does not depend on $\theta$, so that
    \begin{equation}
        \label{eq:sup_below_T_general}
        \sup_{\theta \in B \cap \sTheta} M_n(\theta) \leq T_{r, \ell}
    \end{equation}
    pointwise on the underlying probability space.

    \emph{Step 6: the limit of the upper bound.} We use the strong law of large numbers
    for quasi-integrable variables: if $X, X_1, X_2, \ldots$ are iid with values in $[-\infty, \infty)$ and with $\ex\sbr{X_+} < \infty$, then
    \begin{equation}
        \label{eq:LLN_quasi_general}
        \frac1\ell \sum_{i=1}^\ell X_i \to \ex\sbr{X} \in [-\infty, \infty)
        \mtext{almost surely, as} \ell \to \infty,
    \end{equation}
    the limit being $-\infty$ precisely when $\ex\sbr{X_-} = \infty$. The case $\ex\sbr{X_-} < \infty$ is the ordinary strong law; the case $\ex\sbr{X_-} = \infty$ is \citet[Theorem~2.4.5]{durrett2019probability}, applied to $-X$.

    Apply \Cref{eq:LLN_quasi_general} to each of the $q$ sequences $\rbr{\varphi^{(k)}_{B(\theta_a, \rho_a)}\rbr{\bZ_i^{(k)}}}_{i \in \naturals}$, which are iid by Step~4 and have integrable positive part by \Cref{cond:wald_integrability}, since $\rho_a \leq \delta_{\theta_a}$. As $k \ell / n_r(\ell) \to 1$ and $r \rbr{C_r\rbr{\bZ_0^{(r)}}}_+ / n_r(\ell) \to 0$ almost surely, the first term of \Cref{eq:T_definition_general} vanishes in the limit and the second, being a maximum of finitely many convergent sequences, converges to the maximum of their limits. Hence, with $c$ as in \Cref{eq:c_definition_general},
    \begin{equation}
        \label{eq:T_limit}
        T_{r, \ell} \to c < \mu_0
        \mtext{almost surely, as} \ell \to \infty.
    \end{equation}

    \emph{Step 7: near-maximizers.} Define, for $n \in \naturals$, the gap
    \begin{equation}
        \label{eq:varrho_definition}
        \varrho_n :=
        \begin{dcases}
            0 & \text{if } M_n\rbr{\est_n} \geq M_n\rbr{\theta_0}, \\
            M_n\rbr{\theta_0} - M_n\rbr{\est_n} & \text{otherwise,}
        \end{dcases}
    \end{equation}
    with values in $[0, \infty]$. Since $m$ takes its values in $\reals \cup \cbr{-\infty}$, the criterion $M_n\rbr{\theta_0}$ is never $+\infty$, so that the difference in the second line of \Cref{eq:varrho_definition} is well defined; it equals $+\infty$ exactly when $M_n\rbr{\est_n} = -\infty$ while $M_n\rbr{\theta_0}$ is finite. By construction $M_n\rbr{\est_n} \geq M_n\rbr{\theta_0} - \varrho_n$, and $\varrho_n$ is the smallest $[0, \infty]$-valued random variable with that property; the near-maximizer hypothesis of \Cref{thm:wald_blocks} therefore says precisely that $\varrho_n = o_{\pr}(1)$.

    On the event $\cbr{\est_n \in B}$, which by Step~2 is the event $\cbr{\est_n \in B \cap \sTheta}$, we have $M_n\rbr{\est_n} \leq \sup_{\theta \in B \cap \sTheta} M_n(\theta) \leq T_{r, \ell}$ by \Cref{eq:sup_below_T_general}, whence the inclusion
    \begin{equation}
        \label{eq:wald_inclusion}
        \cbr{\est_{n_r(\ell)} \in B}
        \subseteq
        \cbr{T_{r, \ell} \geq M_{n_r(\ell)}\rbr{\theta_0} - \varrho_{n_r(\ell)}} ,
    \end{equation}
    in which
    every quantity is measurable.

    The variable $\mm{\theta_0}{Z}$ is integrable by Step~1, so that $M_n\rbr{\theta_0} \to \mu_0$ almost surely by the law of large numbers, while $\varrho_n = o_{\pr}(1)$ by assumption. In view of \Cref{eq:T_limit}, the probability of the event on the right-hand side of \Cref{eq:wald_inclusion} therefore tends to zero, which is \Cref{eq:subsequence_goal}.

    Since \Cref{eq:subsequence_goal} holds for each of the $k$ values of $r$, and since these $k$ subsequences exhaust all $n \geq 2k$ by Step~4, we conclude that $\pr[\est_n \in B] \to 0$ as $n \to \infty$, that is,
    \[
        \pr\sbr{d\rbr{\est_n, \mTheta} \geq \eps \; \text{ and } \; \est_n \in K} \to 0,
        \mtext{as} n \to \infty . \qedhere
    \]
\end{proof}

\section{\texorpdfstring{Proofs for \Cref{sec:consistency}}{Proofs for Section 4, Consistency}}
\label{sec:proofs_consistency}
    \begin{lproof}[Proof of \Cref{lem:USC}]
        Let us analyze the behavior of sequences $\bmeta_k \to \bmeta$ according to the location of $\bmeta$ in the extended parameter space $\overline \Eta$, first for sequences inside $\Eta$ and then for sequences in $\overline \Eta \setminus \Eta$.
        The hypothesis $\y \notin \operatorname{col}(\oX)$ can be rewritten as $(\by, \oX) \notin \cN_\infty$ and also as $\operatorname{d}(\oX, \by) > 0$ in the notation of \Cref{eq:null_set,eq:distance}.

        Both the fourth case of Step~1 and step~(ii) of Step~2 rest on the behavior of the supremum of $\tm$ over its second coordinate. For every sequence $\alpha_k \to \alpha$ in $\Rpos$,
        \begin{equation}
            \label{eq:horizon_sup_usc}
            \limsup_{k \to \infty} \; \sup_{\bb \in \Rp} \tm((\alpha_k, \bb) \mid \by, \oX)
            \leq
            \sup_{\bb \in \Rp} \tm((\alpha, \bb) \mid \by, \oX).
        \end{equation}
        Indeed, the supremum over the second coordinate is exactly the profile of \Cref{eq:profile} at block size $l$,
        \begin{equation}
            \label{eq:om_horizon_is_S}
            \sup_{\bb \in \Rp} \tm\rbr{(\alpha, \bb) \mid \by, \oX} = S(\alpha \mid \by, \oX) ,
            \qquad \alpha \in \Rpos ,
        \end{equation}
        and part~(i) of \Cref{lem:profile} says that $S(\cdot \mid \by, \oX)$ is finite and continuous on $\Rpos$. As the sequence $\alpha_k \to \alpha$ is arbitrary, \Cref{eq:horizon_sup_usc} follows.

        \emph{Step 1: sequences in $\Eta$.} Consider sequences $\bmeta_k \in \Eta$ such that $\bmeta_k \to \bmeta$. The four cases below exhaust the possible positions of the limit $\bmeta$ in $\overline \Eta$, the second and the third leaving the second coordinate of $\bmeta$ free:
        \begin{itemize}[leftmargin=*]
        \item $\bmeta \in \Eta$

            Since, for fixed $(y,\bx)$, the map $\bmeta \mapsto m(\bmeta\mid y,\bx)$ is continuous on $\Eta$, it follows that
            $\tm(\bmeta\mid \by,\oX)$ is continuous on $\Eta$ as well. Hence $\om(\cdot\mid \by,\oX)$ is continuous, and also upper-semicontinuous, in the interior of $\overline \Eta$. This holds trivially even when considering sequences $\bmeta_k \in \overline \Eta$, since $\bmeta_k \to \bmeta$ implies that there is some $K_0 \in \naturals$ such that $\bmeta_k \in \Eta$ for every $k>K_0$.

        \item $\bmeta \in \overline \Eta \cap \{\alpha = 0\}$

            Note that $g(z)=z-\exp(z)\le -1$ for all $z\in\reals$. Therefore, $m(\bmeta_k\mid y,\bx) \le \ln\alpha_k-1.$ Hence, $m(\bmeta_k\mid y, \bx)\to -\infty$ whenever $\alpha_k\to 0$, regardless of the behavior of $\bbeta_k$. It follows that $\lim_{k\to\infty}\tm(\bmeta_k\mid \by,\oX)= \om(\bmeta\mid \by,\oX) = -\infty$.

        \item $\bmeta \in \overline \Eta \cap \{\alpha = \infty\}$

            Then $\alpha_k \to \infty$. Since $g(z) = z - e^z \leq -\abr{z}$, we have
            \begin{align*}
                \limsup_{k\to \infty} \tm(\bmeta_k \mid \by, \oX) & \leq \limsup_{k \to \infty} \rbr{\ln \alpha_k - \alpha_k \frac1l \sum_{i=1}^l \abr{\bbeta_k^\tr \bx_i - \ln y_i}} \\
                & \leq \limsup_{k \to \infty} \rbr{\ln \alpha_k - \inf_{\bb \in \Rp}\alpha_k \frac1l \sum_{i=1}^l \abr{\bb^\tr \bx_i - \ln y_i}}\\
                & = \limsup_{k \to \infty} \Bigg(\ln \alpha_k - \frac{\alpha_k}l \underbrace{\inf_{\bb \in \Rp} \norm{\oX \,\bb - \y }_1}_{= \operatorname{d}(\oX, \by)}\Bigg).
            \end{align*}
            Since $\operatorname{d}(\oX, \by) > 0$ by hypothesis, the $\limsup$ is equal to $-\infty = \om(\bmeta\mid \by,\oX)$.

        \item $\bmeta=(\alpha,\bgamma)$ with $\alpha\in\Rpos$ and $\bgamma\in \hzn(\Rp)$

            Let $\bgamma=\dir\bbeta$. Consider any sequence $\bmeta_k \in \Eta$ such that $\bmeta_k = (\alpha_k, \bbeta_k) \to (\alpha,\bgamma)$. Since $\bbeta_k \to \bgamma$, there exists $\lambda_k \downarrow 0$ such that $\lambda_k \bbeta_k \to \bbeta$. In other words, $\bbeta_k = \lambda_k^{-1}{\bbeta} + \bb_k$, for some $\bb_k = o(\lambda_k^{-1})$. Consider $\oX$ such that $\bx_i \not \in \cM_{\bgamma}$ for at least one $i \in \cbr{1, \ldots, l}$. Write
            \begin{align*}
            \abr{r_i(\bbeta_k)} & =\abr{\ln\sigma(\bx_i)-\ln y_i} \\
            & =\abr{\bbeta_k^\tr \bx_i-\ln y_i} \\
            & = \big|\lambda_k^{-1}\underbrace{{\bbeta}^\tr \bx_i}_{\neq 0} + \bb_k^\tr \bx_i - \ln y_i\big|
             \to \infty, \mtext{as $k \to \infty$.}
            \end{align*}
            Since
            \begin{equation*}
                \tm(\bmeta_k\mid \by,\oX) \leq
                \ln\alpha_k-\frac{\alpha_k}l\sum_{i=1}^l\abr{r_i(\bbeta_k)},
            \end{equation*}
            it follows that $\tm(\bmeta_k\mid \by,\oX) \to -\infty$.

            Now consider the case in which $\cbr{\bx_1, \ldots, \bx_l} \subseteq \cM_\bgamma$, so that the second line of the definition of $\om$ applies at $\bmeta$.
            Bounding $\tm$ by its supremum over the second coordinate and then invoking \Cref{eq:horizon_sup_usc}, we find
            \begin{align*}
                \limsup_{k \to \infty} \tm(\bmeta_k \mid \by, \oX)
                & \leq \limsup_{k \to \infty} \sup_{\bb \in \Rp} \tm((\alpha_k, \bb) \mid \by, \oX) \\
                & \leq \sup_{\bb \in \Rp} \tm((\alpha, \bb) \mid \by, \oX)
                = \om(\bmeta \mid \by, \oX) .
            \end{align*}
        \end{itemize}

        \emph{Step 2: sequences in $\overline \Eta \setminus \Eta$.}
        Throughout, $(\by, \oX)$ is as in the statement, so that $\operatorname{d}(\oX, \by) > 0$ by hypothesis. Step~1 has settled the sequences lying in $\Eta$. An arbitrary sequence in $\overline \Eta$ splits into at most two subsequences, one in $\Eta$ and one in $\overline \Eta \setminus \Eta$, both converging to the same limit, and a limit superior along a sequence is the largest of the limits superior along the members of any finite partition of the index set. It therefore remains to consider sequences $\bmeta_k \in \overline \Eta \setminus \Eta$. Since $\Rpos$ is open in $[0, \infty]$ and $\Rp$ is open in $\csm(\Rp)$,
        the set $\Eta$ is open in $\overline \Eta$ and its complement is closed, so that such a sequence has its limit in $\overline \Eta \setminus \Eta$ as well.

        \emph{(i) Two regions.}
        For $\bgamma = \dir \bb \in \hzn(\Rp)$, we have
        \[
            \cbr{\bx_1, \ldots, \bx_l} \subseteq \cM_\bgamma
            \iff
            \bb^\tr \bx_i = 0 \mtext{for all} i = 1, \ldots, l
            \iff
            \oX \bb = \0 ,
        \]
        so the horizon points at which the second line of the definition of $\om$ applies are exactly those of $\hzn(\ker \oX)$, the horizon of the kernel of $\oX$. In particular, that set is empty if and only if $\oX$ has full column rank. It is closed in $\hzn(\Rp)$: if $\bgamma_k = \dir \bb_k \in \hzn(\ker \oX)$ and $\bgamma_k \to \bgamma = \dir \bb$, then $\lambda_k \bb_k \to \bb$ for some scalars $\lambda_k > 0$, and therefore $\oX \bb = \lim_{k \to \infty} \oX (\lambda_k \bb_k) = \0$, that is, $\bgamma \in \hzn(\ker \oX)$.
        Therefore, if a convergent sequence of directions $\bgamma_k$ satisfies $\cbr{\bx_1, \ldots, \bx_l} \subseteq \cM_{\bgamma_k}$ for all large $k$, then also $\cbr{\bx_1, \ldots, \bx_l} \subseteq \cM_{\bgamma}$ for the limit $\bgamma$ of the sequence.

        The boundary is
        \[
            \overline \Eta \setminus \Eta
            =
            \rbr{\cbr{0, \infty} \times \csm(\Rp)} \cup \rbr{\Rpos \times \hzn(\Rp)} ,
        \]
        and by \Cref{lem:USC} and the previous paragraph it is the disjoint union of two regions:
        \begin{itemize}[leftmargin=*]
            \item $\Rpos \times \hzn(\ker \oX)$, on which $\om(\bmeta \mid \by, \oX) = S(\alpha \mid \by, \oX)$, finite and continuous in $\alpha$ by \Cref{lem:profile}\,(i), and not depending on $\bgamma$ at all, the supremum over the whole of $\Rp$ having erased the direction;
            \item the remainder, on which $\om(\cdot \mid \by, \oX) = -\infty$.
        \end{itemize}

        \emph{(ii) One uniform bound.} Extend $S(\cdot \mid \by, \oX)$ to $[0, \infty]$ by setting $S(0 \mid \by, \oX) := S(\infty \mid \by, \oX) := -\infty$. Then
        \begin{equation}
            \label{eq:om_below_S}
            \om(\bmeta \mid \by, \oX) \leq S(\alpha \mid \by, \oX),
            \mtext{for every} \bmeta = (\alpha, \bt) \in \overline \Eta ,
        \end{equation}
        and $S(\cdot \mid \by, \oX) : [0, \infty] \to [-\infty, \infty)$ is continuous. For the bound, run through the three lines of the definition of $\om$ in \Cref{lem:USC}: on $\Eta$, $\tm((\alpha, \bbeta) \mid \by, \oX) \leq \sup_{\bb \in \Rp} \tm((\alpha, \bb) \mid \by, \oX) = S(\alpha \mid \by, \oX)$ by \Cref{eq:om_horizon_is_S}; on $\Rpos \times \hzn(\ker \oX)$ there is equality; and elsewhere the left-hand side is $-\infty$. The continuity on $\Rpos$ follows from \Cref{lem:profile}\,(i) and the continuity at the two endpoints is the statement that $S(\alpha \mid \by, \oX) \to -\infty$ as $\alpha \downarrow 0$ and as $\alpha \to \infty$, see \Cref{lem:profile}~(ii) and~(iii).

        \emph{(iii) Sequences.} Let $\bmeta_k = (\alpha_k, \bt_k) \to \bmeta = (\alpha, \bt)$, all of them in $\overline \Eta \setminus \Eta$. Convergence in the product space implies $\alpha_k \to \alpha$ in $[0, \infty]$, so that \Cref{eq:om_below_S} and the continuity of $S(\cdot \mid \by, \oX)$ give
        \begin{equation}
            \label{eq:limsup_below_S}
            \limsup_{k \to \infty} \om(\bmeta_k \mid \by, \oX)
            \leq
            \lim_{k \to \infty} S(\alpha_k \mid \by, \oX)
            =
            S(\alpha \mid \by, \oX) .
        \end{equation}
        The limit lies in exactly one of three positions, since $\bmeta \not\in \Eta$ forces $\bt$ to be a horizon point as soon as $\alpha \in \Rpos$, and that horizon point either lies in $\hzn(\ker \oX)$ or does not. In the first two positions, the proof is finished by \Cref{eq:limsup_below_S}:
        \begin{enumerate}
            \item $\alpha \in \cbr{0, \infty}$: then $\om(\bmeta \mid \by, \oX) = -\infty = S(\alpha \mid \by, \oX)$, whatever $\bt$ is;
            \item $\alpha \in \Rpos$ and $\bt = \bgamma \in \hzn(\ker \oX)$: then $\om(\bmeta \mid \by, \oX) = S(\alpha \mid \by, \oX)$;
            \item $\alpha \in \Rpos$ and $\bt = \bgamma \in \hzn(\Rp) \setminus \hzn(\ker \oX)$: then $\om(\bmeta \mid \by, \oX) = -\infty$ while $S(\alpha \mid \by, \oX)$ is finite, so that \Cref{eq:limsup_below_S} is not enough.
        \end{enumerate}
        In the third position, there exists $k_0$ such that $\bmeta_k \not\in \Rpos \times \hzn(\ker \oX)$ for every $k > k_0$, since $\hzn(\Rp) \setminus \hzn(\ker \oX)$ is open relative to $\hzn(\Rp)$.
        For $k > k_0$, the point $\bmeta_k$ therefore lies in the second of the two regions, so that $\om(\bmeta_k \mid \by, \oX) = -\infty$,
        which completes the proof.
    \end{lproof}

    \begin{lproof}[Proof of \Cref{lem:integrable_boundary}]
        Let $\Y := \rbr{\ln Y_1, \ldots, \ln Y_l}^\tr$ and recall the null set $\cN_\infty$ of \Cref{eq:null_set}, which satisfies $\pr_{\bmeta_\dagger}\sbr{(\bY, \mX) \in \cN_\infty} = 0$ by part~(i) of \Cref{lem:design_regularity}.

        Let $\bmeta\in \overline\Eta\setminus\Eta$. If $\alpha\in\{0,\infty\}$, then $\om(\bmeta\mid \bY,\mX)=-\infty$ almost surely by definition, and hence
        \[
        \ex_{\bmeta_\dagger}\sbr{\om(\bmeta\mid \bY,\mX)}=-\infty.
        \]
        Now let $\bmeta=(\alpha,\bgamma)$ with $\alpha\in\Rpos$ and $\bgamma\in \hzn(\Rp)$. Then
        \[
            \om(\bmeta\mid \bY,\mX) =
            \begin{dcases}
                \sup_{\bb \in \Rp} \tm((\alpha, \bb) \mid \bY, \mX), & \text{ if } \cbr{\bX_1, \ldots, \bX_l} \subseteq \cM_\bgamma,\\
                -\infty, & \text{ if } \cbr{\bX_1, \ldots, \bX_l} \not \subseteq \cM_\bgamma.
            \end{dcases}
        \]
        From \Cref{cond:identifiability} we have $\pr_{\bmeta_\dagger}(\bX_i \in \cM_\bgamma) < 1$, and therefore
        \[
        \pr_{\bmeta_\dagger}\sbr{\cbr{\bX_1, \ldots, \bX_l} \subseteq \cM_\bgamma}
        =
        \pr_{\bmeta_\dagger}\sbr{\bX_1 \in \cM_\bgamma}^l < 1.
        \]
        Hence
        \[
        \pr_{\bmeta_\dagger}\sbr{\om(\bmeta\mid \bY,\mX)=-\infty}>0,
        \]
        so that the negative part of $\om(\bmeta\mid \bY,\mX)$ has infinite expectation. It remains to show that its positive part is integrable; without that, the expectation would be of the form $\infty - \infty$ and would not be defined. By \Cref{eq:objective_function,eq:g_bounds}, the block average is uniformly bounded in $\bb$,
        \[
            \tm((\alpha, \bb) \mid \bY, \mX)
            \leq
            \ln \alpha - 1
            \mtext{for every} \bb \in \Rp,
        \]
        It follows that
        $\ex_{\bmeta_\dagger}\sbr{\rbr{\om(\bmeta\mid \bY,\mX)}_+} < \infty$, whatever the distribution of $\bX$. The extension $\om(\bmeta \mid \bY, \mX)$ is thus quasi-integrable, and its expectation under $\bmeta_{\dagger}$ is $-\infty$.
    \end{lproof}

    \begin{lproof}[Proof of \Cref{lem:measurability}]
        Throughout, $(\by, \oX)$ ranges over $\Rpos^l \times \reals^{l \times p}$, and $\bx_1^\tr, \ldots, \bx_l^\tr$ denote the rows of $\oX$. Two facts are used repeatedly.

        First, by \Cref{eq:objective_function,eq:l_objective_function},
        \begin{equation}
            \label{eq:tilde_m_continuous}
            \tm\rbr{(\alpha, \bb) \mid \by, \oX}
            =
            \frac1l \sum_{i=1}^l \sbr{\ln \alpha + g\rbr{\alpha \rbr{\bb^\tr \bx_i - \ln y_i}}}
        \end{equation}
        is a continuous, real-valued function of $\rbr{(\alpha, \bb), \by, \oX}$ on $\Eta \times \Rpos^l \times \reals^{l \times p}$. Note that $\Eta$ is an open subset of $\overline\Eta$ and that the topology it inherits is the usual one of $\Rpos \times \Rp$.

        Second, for a horizon point $\bgamma = \dir \bbeta$ with $\bbeta$ in the Euclidean unit sphere $\sphere := \cbr{\bb \in \Rp : \norm{\bb} = 1}$,
        \begin{equation}
            \label{eq:hyperplane_as_kernel}
            \cbr{\bx_1, \ldots, \bx_l} \subseteq \cM_\bgamma
            \iff
            \bbeta^\tr \bx_i = 0 \text{ for } i = 1, \ldots, l
            \iff
            \oX \bbeta = \0 .
        \end{equation}
        Every horizon point has exactly one such representative $\bbeta$, and $\bgamma \mapsto \bbeta$ is a homeomorphism of $\hzn(\Rp)$ onto $\sphere$: in the hemispherical model of \Cref{subsec:compactification} it is the identification of the rim of $H_p$ with $\sphere$. We use it to read $\hzn(\Rp)$ as $\sphere$ without further comment.

        \emph{Step 1: two reductions.}
        If $A = \bigcup_{j \in J} A_j$ with $J$ finite or countable, then
        \[
            \sup_{\bmeta \in A} \om(\bmeta \mid \by, \oX)
            =
            \sup_{j \in J} \; \sup_{\bmeta \in A_j} \om(\bmeta \mid \by, \oX),
        \]
        and a countable supremum of Borel maps with values in $[-\infty, +\infty]$ is Borel. It is therefore enough to prove the measurability of the suprema over the pieces $A_j$, which need not be open. Moreover, $\overline\Eta$ is the disjoint union
        \begin{equation}
            \label{eq:three_regions}
            \overline\Eta
            =
            \Eta
            \; \cup \;
            \rbr{\cbr{0, \infty} \times \csm(\Rp)}
            \; \cup \;
            \rbr{\Rpos \times \hzn(\Rp)},
        \end{equation}
        the three regions being those on which the three branches of $\om$ in \Cref{lem:USC} apply.

        \emph{Step 2: suprema over subsets of $\Eta$.}
        We claim that
        \begin{equation}
            \label{eq:sup_over_Eta_subset}
            (\by, \oX) \mapsto \sup_{\bmeta \in A} \tm(\bmeta \mid \by, \oX)
            \mtext{is Borel for every} A \subseteq \Eta .
        \end{equation}
        Indeed, $\Eta$ is a separable metric space, so $A$ has a countable dense subset $D$. By the continuity in \Cref{eq:tilde_m_continuous}, $\sup_{\bmeta \in A} \tm(\bmeta \mid \by, \oX) = \sup_{\bmeta \in D} \tm(\bmeta \mid \by, \oX)$ for every $(\by, \oX)$, and each of the maps $(\by, \oX) \mapsto \tm(\bmeta \mid \by, \oX)$, $\bmeta \in D$, is continuous. Step~1 applies.

        \emph{Step 3: unions of kernels over a compact set of directions.}
        For $C \subseteq \sphere$, put
        \[
            M_C := \cbr{\oX \in \reals^{l \times p} : \oX \bbeta = \0 \text{ for some } \bbeta \in C} .
        \]
        We claim that $M_C$ is closed whenever $C$ is compact. Let $\oX_k \in M_C$ with $\oX_k \to \oX$, and pick $\bbeta_k \in C$ with $\oX_k \bbeta_k = \0$. By compactness there is a subsequence $\bbeta_{k_j} \to \bbeta \in C$, and then $\oX \bbeta = \lim_{j \to \infty} \oX_{k_j} \bbeta_{k_j} = \0$ by continuity of $(\oX, \bbeta) \mapsto \oX \bbeta$. Hence $\oX \in M_C$.

        \emph{Step 4: proof of part~(i).}
        Fix $\bmeta \in \overline\Eta$ and distinguish the three regions in \Cref{eq:three_regions}. If $\bmeta \in \Eta$, then $\om(\bmeta \mid \cdot, \cdot) = \tm(\bmeta \mid \cdot, \cdot)$ is continuous by \Cref{eq:tilde_m_continuous}. If $\bmeta \in \cbr{0, \infty} \times \csm(\Rp)$, then $\om(\bmeta \mid \cdot, \cdot) \equiv -\infty$. Finally, let $\bmeta = (\alpha, \bgamma)$ with $\alpha \in \Rpos$ and $\bgamma \in \hzn(\Rp)$, represented by $\bbeta \in \sphere$. By \Cref{eq:hyperplane_as_kernel},
        \[
            \om(\bmeta \mid \by, \oX)
            =
            \begin{dcases}
                \sup_{\bb \in \Rp} \tm\rbr{(\alpha, \bb) \mid \by, \oX},
                & \text{if } \oX \in M_{\cbr{\bbeta}},\\
                -\infty,
                & \text{otherwise.}
            \end{dcases}
        \]
        The supremum is Borel in $(\by, \oX)$ by \Cref{eq:sup_over_Eta_subset} applied to $A = \cbr{\alpha} \times \Rp$, and $M_{\cbr{\bbeta}}$ is closed by Step~3. The map is thus Borel, being equal to a Borel map on the Borel set $\Rpos^l \times M_{\cbr{\bbeta}}$ and constant on its complement.

        \emph{Step 5: proof of part~(ii), the two easy regions.}
        Let $U \subseteq \overline\Eta$ be open. Splitting $U$ along \Cref{eq:three_regions} and invoking Step~1, it suffices to treat the three traces separately. On $U \cap \Eta$ the extension coincides with $\tm$, so \Cref{eq:sup_over_Eta_subset} applies. On $U \cap \rbr{\cbr{0, \infty} \times \csm(\Rp)}$ the extension is identically $-\infty$, so the supremum is the constant $-\infty$, whether or not that trace is empty. What remains is
        \[
            \Uhzn := U \cap \rbr{\Rpos \times \hzn(\Rp)},
        \]
        a relatively open subset of $\Rpos \times \sphere$.

        \emph{Step 6: reduction of $\Uhzn$ to countably many boxes.}
        Fix a countable dense subset $\cbr{\bv_1, \bv_2, \ldots}$ of $\sphere$ and let $\cF$ be the countable family of all sets of the form
        \[
            (a, b) \times \cbr{\bbeta \in \sphere : \norm{\bbeta - \bv_j} \leq r},
            \qquad
            0 < a < b, \; r > 0 \text{ rational}, \; j \in \naturals .
        \]
        Each member of $\cF$ is a product $I \times C$ of an open interval $I \subseteq \Rpos$ and a compact set $C \subseteq \sphere$. We claim that $\Uhzn$ is the union of those members of $\cF$ that are contained in $\Uhzn$; by Step~1 it then suffices to prove the measurability of $\sup_{\bmeta \in I \times C} \om(\bmeta \mid \by, \oX)$ for a single such box.

        Let $(\alpha, \bbeta) \in \Uhzn$. Since $\Uhzn$ is relatively open in the product $\Rpos \times \sphere$, there is $\eps \in (0, \alpha)$ with
        \[
            \cbr{\alpha' \in \Rpos : \abr{\alpha' - \alpha} < \eps}
            \times
            \cbr{\bbeta' \in \sphere : \norm{\bbeta' - \bbeta} < \eps}
            \subseteq \Uhzn .
        \]
        Choose rationals $a, b$ with $\alpha - \eps < a < \alpha < b < \alpha + \eps$, a rational $r \in (0, \eps/2)$ and an index $j$ with $\norm{\bv_j - \bbeta} < r$. The corresponding member of $\cF$ contains $(\alpha, \bbeta)$, and it is contained in $\Uhzn$, because $\norm{\bbeta' - \bv_j} \leq r$ implies $\norm{\bbeta' - \bbeta} \leq 2r < \eps$.

        \emph{Step 7: the supremum over a box.}
        Let $I \times C \in \cF$ and set
        \[
            T_I(\by, \oX) := \sup_{(\alpha, \bb) \in I \times \Rp} \tm\rbr{(\alpha, \bb) \mid \by, \oX} .
        \]
        Let $(\alpha, \bgamma)$ be a point of $I \times C$ and let $\bbeta \in \sphere$ represent $\bgamma$. By \Cref{eq:hyperplane_as_kernel}, the value $\om\rbr{(\alpha, \bgamma) \mid \by, \oX}$ equals $\sup_{\bb \in \Rp} \tm\rbr{(\alpha, \bb) \mid \by, \oX}$ if $\oX \bbeta = \0$, and $-\infty$ otherwise; crucially, in the first case the value does not depend on $\bgamma$. Hence, if no $\bbeta \in C$ annihilates $\oX$, every term of the supremum is $-\infty$, while if at least one does, the supremum is unchanged if $\bgamma$ is dropped from the index set altogether:
        \[
            \sup_{\bmeta \in I \times C} \om(\bmeta \mid \by, \oX)
            =
            \begin{dcases}
                T_I(\by, \oX), & \text{if } \oX \in M_C,\\
                -\infty, & \text{otherwise.}
            \end{dcases}
        \]
        By \Cref{eq:sup_over_Eta_subset} applied to $A = I \times \Rp$, the map $T_I$ is Borel, and $M_C$ is closed by Step~3. As in Step~4, the map is therefore Borel, which completes the proof.
    \end{lproof}

    \begin{lproof}[Proof of \Cref{lem:integrability}]
        Throughout, we work on the event $\cbr{(\bY, \mX) \notin \cN_\infty}$, which has probability one by \Cref{eq:null_set_is_null}, so that \Cref{lem:USC} and the bounds of \Cref{lem:profile} apply to the realization at hand.

        Write $\Y := \rbr{\ln Y_1, \ldots, \ln Y_l}^\tr$ for the vector of log-responses, $\bY$ and $\mX$ being as in \Cref{eq:block_notation}, and put
        \[
            R := \inf_{\bb \in \Rp} \norm{\Y - \mX \bb}_2
        \]
        for the Euclidean distance from $\Y$ to the column space of $\mX$. On the event we work on, $\Y \notin \operatorname{col}(\mX)$ and hence $R > 0$.

        \emph{Step 1: a bound in terms of $R$.}
        We have
        \begin{align*}
            \sup_{\bmeta \in \overline\Eta} \om(\bmeta\mid \bY,\mX)
            & \leq \sup_{\alpha \in \Rpos} S(\alpha \mid \bY, \mX) \\
            & \leq \sup_{\alpha \in \Rpos}\rbr{\ln\alpha - \frac\alpha{l} \inf_{\bb \in \Rp}\norm{\Y - \mX \bb}_1} \tag{$*$}\\
            & \leq \sup_{\alpha \in \Rpos}\rbr{\ln\alpha - \frac{\alpha R}{l}} , \tag{$**$}
        \end{align*}
        the first step by \Cref{eq:om_below_S}, the endpoint values $S(0 \mid \bY, \mX) = S(\infty \mid \bY, \mX) = -\infty$ leaving only $\alpha \in \Rpos$; the second by part~(iii) of \Cref{lem:profile} and \Cref{eq:distance}; and the third because $\norm{\,\cdot\,}_1 \geq \norm{\,\cdot\,}_2$. The infima over $\bb$ in $(*)$ and $(**)$ correspond to the minimization problem with respect to the $L_1$ and $L_2$ norms, respectively. Since $R > 0$, the supremum in $(**)$ is attained at the finite value $\alpha_{\text{opt}} = l / R$, so that
        \begin{equation}
            \label{eq:sup_below_R}
            \sup_{\bmeta \in \overline\Eta} \om(\bmeta \mid \bY, \mX) \leq \ln l - 1 - \ln R .
        \end{equation}

        \emph{Step 2: a lower bound for $R$.}
        Under \Cref{eq:model}, the variables $G_i := \ad\rbr{\ln Y_i - \bd^\tr \bX_i}$, for $i = 1, \ldots, l$, are iid standard Gumbel and independent of $\mX$: the conditional law of $\ad \ln\rbr{Y_i / \sigma(\bX_i)}$ given $\bX_i = \bx$ is $\Gumbel(0,1)$ whatever $\bx$ is. Collecting them in $\bG = (G_1, \ldots, G_l)^\tr$, we have $\Y = \mX \bd + \ad^{-1} \bG$.

        Fix a realization $\oX$ of $\mX$. The solution to the $L_2$ minimization problem is the orthogonal projection of $\Y$ onto the orthogonal complement of the column space of $\oX$, denoted $\operatorname{col}(\oX)^\perp$, and represented in matrix form by the $l \times l$ projection matrix $\mQ_{\oX}$. Since $\Y = \oX \bd + \ad^{-1} \bG$, we have $\mQ_{\oX} \Y = \ad^{-1} \mQ_{\oX} \bG$, and therefore $R = \norm{\mQ_{\oX} \Y}_2 = \|\ad^{-1} \mQ_{\oX} \bG\|_2$. The projection matrix $\mQ_{\oX}$ can be represented as $\mO \mLambda \mO^\tr$, for $\mO$ orthogonal and $\mLambda$ a diagonal matrix containing the eigenvalues of $\mQ_{\oX}$, which are $\lambda_1 = \ldots = \lambda_{l - \rank(\oX)} = 1$ and $\lambda_j = 0$ for all $j > l - \rank(\oX)$; note that $l - \rank(\oX) \geq l - p \geq 1$, since $l \geq p+1$. Let $\bW := \mO^\tr \bG$. Then
        \begin{equation*}
            \norm{\ad^{-1} \mQ_{\oX} \bG}^2_2
            = \ad^{-2} \bG^\tr \mQ_{\oX} \bG
            = \ad^{-2} \bW^\tr \mLambda \bW
            = \ad^{-2} \sum_{i=1}^{l - \rank(\oX)} W_i^2
            \geq \ad^{-2} W_1^2 ,
        \end{equation*}
        the first equality because $\mQ_{\oX}^\tr \mQ_{\oX} = \mQ_{\oX}$ and the second because $\mQ_{\oX} = \mO \mLambda \mO^\tr$. The first row of $\mO^\tr$ is the transpose of a unit vector $\bw \in \sphere[l-1]$, so that $W_1 = \bw^\tr \bG$ and
        \begin{equation}
            \label{eq:R_below_projection}
            R^2 \geq \ad^{-2} \abr{\bw^\tr \bG}^2 ,
            \mtext{so that} R \geq \ad^{-1} \abr{\bw^\tr \bG} .
        \end{equation}
        Both $\mO$ and $\bw$ depend on the realization $\oX$, but the bound of Step~3 below is uniform in all unit vectors in $\sphere[l-1]$.

        \emph{Step 3: a bound uniform in the direction.} Let $\bv \in \sphere[l-1]$ be arbitrary. We claim that
        \begin{equation}
            \label{eq:log_projection_integrable}
            \ex\sbr{\rbr{- \ln \abr{\langle \bv, \bG \rangle}}_+} \leq 2 K ,
            \mtext{where} K := l \, \norm{f^{G}}_\infty ,
        \end{equation}
        with $f^{G}$ the standard Gumbel density and $\norm{\,\cdot\,}_\infty$ the supremum norm. Choose $i \in \cbr{1,\ldots,l}$ with $\abr{v_i} = \max_j \abr{v_j}$ and define $\tilde U_i := \sum_{j \neq i} {v_j G_j}$, independent of $G_i$. For the Gumbel distribution, $\norm{f^{a G_i}}_\infty = \abr{a}^{-1}\norm{f^{G}}_\infty < \infty$ for $a \neq 0$, and $\sum_{j=1}^l v_j^2 = 1$ implies $\max_j \abr{v_j} \geq l^{-1/2} \geq l^{-1}$,
        from which we derive that $f^{\langle \bv, \bG \rangle}$ is uniformly bounded, since
        \begin{align*}
            f^{\langle \bv, \bG \rangle}(z)  = \rbr{f^{v_i G_i} * f^{\tilde U_i}} (z)
            & = \int_\reals f^{v_i G_i}(t) f^{\tilde U_i}(z - t) \,dt \\
            & \leq \norm{f^{v_i G_i}}_\infty
            \int_\reals f^{\tilde U_i}(z - t) \,dt
            = \frac{\norm{f^{G}}_\infty}{\abr{v_i}} \cdot 1
            \leq l \, \norm{f^{ G}}_\infty = K .
        \end{align*}
        Therefore
        \[
            \ex\sbr{\rbr{- \ln \abr{\langle \bv, \bG \rangle}}_+}
            =
            - \int_{-1}^1 \ln \abr{z} \, f^{\langle \bv, \bG \rangle}(z) \,dz
            \leq
            - K \int_{-1}^1 \ln \abr{z} \,dz
            =
            2 K ,
        \]
        which is \Cref{eq:log_projection_integrable}.

        \emph{Step 4: positive part and expectation.}
        By \Cref{eq:sup_below_R} and $(a+b)_+ \leq a_+ + b_+$,
        \[
            \rbr{\sup_{\bmeta \in \overline\Eta} \om(\bmeta \mid \bY, \mX)}_+
            \leq
            \rbr{\ln l - 1}_+ + \rbr{- \ln R}_+ ,
        \]
        an inequality between random variables, the left-hand side of which is measurable by part~(ii) of \Cref{lem:measurability}. It remains to bound the expectation of the second term. Since $\bG$ is independent of $\mX$ and $R$ is a measurable function of $(\mX, \bG)$, conditioning on $\mX$ fixes the realization $\oX$, and \Cref{eq:R_below_projection,eq:log_projection_integrable} applied at the unit vector $\bw$ belonging to that realization give
        \[
            \ex\sbr{\rbr{- \ln R}_+ \mid \mX = \oX}
            \leq
            \rbr{\ln \ad}_+ + \ex\sbr{\rbr{- \ln \abr{\langle \bw, \bG \rangle}}_+}
            \leq
            \rbr{\ln \ad}_+ + 2K
        \]
        for almost every $\oX$. Taking expectations yields \Cref{eq:wald_integrability_frechet}, since
        \[
            \ex_{\bmeta_\dagger}\sbr{\rbr{\sup_{\bmeta \in \overline\Eta} \om(\bmeta \mid \bY, \mX)}_+}
            \leq
            \rbr{\ln l - 1}_+ + \rbr{\ln \ad}_+ + 2K
            <
            \infty. \qedhere
        \]
    \end{lproof}

    \begin{proof}[Proof of \Cref{thm:consistency}]
        We apply \Cref{thm:wald_blocks} with $\cZ = \Rpos \times \Rp$ and $Z = (Y, \bX)$, with $\sTheta = \Eta$ and $\Theta = \overline\Eta$ metrized by the $d_\infty$ of \Cref{subsec:compactification}, and with the one-observation objective function $m$ of \Cref{eq:objective_function}. A block $\bz^{(l)} \in \cZ^l$ is the pair $(\by, \oX) \in \Rpos^l \times \reals^{l \times p}$ of \Cref{lem:USC}, up to the rearrangement of the coordinates, which is bimeasurable and which we leave implicit.

        Put $k := p+1$. Every $l \in \cbr{k, \ldots, 2k-1}$ then satisfies $l \geq p+1$, so that \Cref{lem:USC} supplies the extension $\om$ of the block average $\tm$ to $\overline\Eta$ at all of those block sizes at once, and not only at one of them.
        That $\bz^{(l)} \mapsto \om\rbr{\bmeta \mid \bz^{(l)}}$ is measurable for every fixed $\bmeta \in \overline\Eta$ is part~(i) of \Cref{lem:measurability}. Hypothesis~(i) of \Cref{thm:wald_blocks} holds by the definition of $\om$, which on $\Eta$ is the block average itself. Hypothesis~(ii) is \Cref{lem:USC} for the upper-semicontinuity, and \Cref{lem:measurability,lem:integrability} for items~(a) and~(b) of \Cref{cond:wald_integrability}. \Cref{lem:USC} is a deterministic statement about a realization with $\y \notin \operatorname{col}(\oX)$; since $l \geq p+1$, \Cref{eq:null_set_is_null} makes that hypothesis hold outside the null set $\cN_\infty$ of \Cref{eq:null_set}, which is the null set that \Cref{def:upper_semicontinuity} allows --- one and the same for every $\bmeta \in \overline\Eta$. Hypothesis~(iii) is \Cref{lem:integrable_boundary}, applied at the block size $k$; \Cref{cond:identifiability}, which that lemma requires, is assumed. Finally, the set $\mTheta$ of \Cref{eq:Theta_0_blocks} is here the set $\Eta_0$ of \Cref{eq:Eta_0}, which \Cref{lem:unique_maximizer} identifies as the singleton $\cbr{\bmeta_\dagger}$; it is in particular nonempty, and the near-maximizer hypothesis of \Cref{thm:wald_blocks} at $\theta_0 = \bmeta_\dagger$ is the hypothesis of \Cref{thm:consistency}.

        The space $(\overline\Eta, d_\infty)$ is compact by \Cref{subsec:compactification}, so that $K := \overline\Eta$ is an admissible choice in \Cref{thm:wald_blocks}.
        Since $\nme$ takes its values in $\Eta \subseteq K$, the event $\cbr{\nme \in K}$ is sure, and since $\mTheta$ is a singleton, $d_\infty\rbr{\nme, \mTheta} = d_\infty\rbr{\nme, \bmeta_\dagger}$. \Cref{thm:wald_blocks} therefore yields
        \[
            \pr\sbr{d_\infty\rbr{\nme, \bmeta_\dagger} \geq \eps} \to 0
            \mtext{as} n \to \infty, \mtext{for every} \eps > 0 ,
        \]
        which is consistency in the metric of the compactification. As $\bmeta_\dagger \in \Eta$ and $\nme$ is $\Eta$-valued, \Cref{eq:metrics_agree} turns this into consistency in the Euclidean metric, $\nme \pto \bmeta_\dagger$.
    \end{proof}

\section{\texorpdfstring{Proofs for \Cref{sec:asymptotic_normality}}{Proofs for Section 5, Asymptotic Normality}}
\label{sec:proofs_normality}

    \begin{proof}[Proof of \Cref{thm:normality}]
Theorem~5.41 in \citet{van2000asymptotic} concerns estimators that are zeros of the score, and the CMLE is one. For every $n \geq p+1$, it maximizes $M_n$ over the open set $\Eta$ almost surely, as recalled in \Cref{sec:asymptotic_normality}, and $M_n$ is differentiable, so that $\sum_{i=1}^n \nabla_\bmeta m(\mle \mid Y_i, \bX_i) = \0$ almost surely. Neither an event of probability zero nor the finitely many sample sizes $n < p+1$ have any bearing on a limit statement in probability. It therefore suffices to check the regularity conditions of that theorem, which are:
        \begin{enumerate}[label=\roman*.]
            \item $m(\bmeta \mid y, \bx)$ is three times continuously differentiable in $\bmeta$ for every $(y, \bx)$;
            \item $\ex[\nabla_\bmeta m(\bmeta_\dagger \mid Y, \bX)] = 0$;
            \item $\ex[\|\nabla_\bmeta m(\bmeta_\dagger \mid Y, \bX)\|^2] < \infty$;
            \item The expected Hessian matrix, $\Hes = \ex\sbr{\nabla_\bmeta^2 m(\bmeta_\dagger \mid Y, \bX)}$, exists and is nonsingular;
            \item For all $\bmeta$ in a neighborhood of $\bmeta_\dagger$, the elements of $\nabla_\bmeta^3 m(\bmeta \mid y, \bx)$ are dominated by a fixed integrable function $\dddot M(y, \bx)$.
        \end{enumerate}

        Define $r_\bbeta(\bx, y) = \bbeta^\tr \bx - \ln y$. Then
        \[m(\bmeta \mid y, \bx) = \ln \alpha + g(\alpha \; r_\bbeta(\bx, y)),\]
        with $g(t) = t - \exp(t)$.
        For every fixed $(y, \bx)\in \Rpos \times \Rp$, the map $\bmeta \mapsto m(\bmeta \mid  y,\bx)$ is infinitely differentiable, since $\ln$ and $g$ are smooth on $\Rpos$ and $\reals$, respectively, and the map $(\alpha,\bbeta)\mapsto \alpha \;r_\bbeta(\bx, y)$ is a polynomial. Hence, \textbf{item i} is satisfied.

        For $\bx \in \Rp$, write $\sigma_\dagger(\bx) := \sigma_{\bd}(\bx) = \exp(\bd^\tr \bx)$ for the scale function at the true parameter and consider the pair of transformations
        \begin{align}
            \label{eq:barz}
            \bar{z}(y, \bx)
            &= \rbr{\frac{y}{\sigma_\dagger(\bx)}}^{-\ad},
            & y \in \Rpos, \\
            \label{eq:bary}
            \bar{y}(z, \bx)
            &= \sigma_\dagger(\bx) \;z^{-1/\ad},
            & z \in \Rpos,
        \end{align}
        which are each other's inverse: $\bar{y}(\bar{z}(y,\bx),\bx) = y$ and $\bar{z}(\bar{y}(z,\bx),\bx) = z$.
        Note that
        \[
            \exp\rbr{\ad r_{\bd}(\bx, y)}
            = \rbr{\frac{y}{\sigma_{\dagger}(\bx)}}^{-\ad} = \bar{z}(y,\bx)
            \mtext{and}
            r_{\bd}(\bx, y)
            = \ad^{-1} \ln \bar{z}(y,\bx).
        \]
        Further,
        \[
            r_{\bbeta}(\bx, y)
            = \Delta_{\bbeta}^\tr \,\bx + r_{\bd}(\bx, y)
            = \Delta_{\bbeta}^\tr \,\bx + \ad^{-1} \ln \bar{z}(y,\bx)
            \mtext{with}
            \Delta_{\bbeta} = \bbeta - \bd.
        \]
        Define
        \begin{align}
            \notag
            q(\bmeta \mid z,\bx)
            &= m(\bmeta \mid \bar{y}(z,\bx), \bx) \\
            \notag
            &= \ln \alpha
            - \ln \sigma_\dagger(\bx)
            + \frac{1}{\ad}\ln z
            + g\rbr{\alpha \Delta_{\bbeta}^\tr \,\bx
            + \frac{\alpha}{\ad} \ln z} \\
            \label{eq:qbmeta}
            &= \ln \alpha
            - \bd^\tr \bx
            + \alpha \Delta_{\bbeta}^\tr \,\bx
            + \frac{1 + \alpha}{\ad} \ln z
            - \exp\rbr{\alpha \Delta_{\bbeta}^\tr \,\bx} z^{\alpha/\ad}.
        \end{align}
        Since $g'(t) = 1 - e^t$ and $ g''(t) = - e^t$, we have
        \begin{equation*}
            \nabla_{\bmeta} \,m(\bmeta \mid y, \bx) =
            \begin{pmatrix}
                \partial_{\alpha} m(\bmeta \mid y, \bx)\\
                \nabla_{\bbeta} m(\bmeta \mid y, \bx)
            \end{pmatrix} =
            \begin{pmatrix}
                \alpha^{-1} + r_\bbeta(\bx, y)\; g'(\alpha \;r_\bbeta(\bx, y)) \\
                \alpha \bx\;  g'(\alpha\;r_\bbeta(\bx, y))
            \end{pmatrix}.
        \end{equation*}
        As, moreover,
        \begin{align*}
            r_\bbeta(\bx, \bar{y}(z,\bx))
            &= \Delta_{\bbeta}^\tr \,\bx
            + \ad^{-1} \ln \bar{z}(\bar{y}(z,\bx),\bx) \\
            &= \Delta_{\bbeta}^\tr \,\bx
            + \ad^{-1} \ln z,
        \end{align*}
        we find
        \begin{align}
            \notag
            \nabla_{\bmeta} \, q(\bmeta \mid z, \bx)
            &=
            \nabla_{\bmeta} \, m(\bmeta \mid \bar{y}(z,\bx),\bx) \\
            \label{eq:score_general}
            &=
            \begin{pmatrix}
                \alpha^{-1} + \rbr{\Delta_{\bbeta}^\tr \,\bx + \frac{1}{\ad} \ln z} \, g'\rbr{\alpha \Delta_{\bbeta}^\tr \,\bx + \frac{\alpha}{\ad} \ln z} \\
                \alpha \bx \, g'\rbr{\alpha \Delta_{\bbeta}^\tr \,\bx + \frac{\alpha}{\ad} \ln z}
            \end{pmatrix}.
        \end{align}
        At the true parameter $\bmeta_{\dagger}$, \Cref{eq:score_general} simplifies to
        \begin{equation}
            \label{eq:log_likelihood_reparametrization}
            \nabla_{\bmeta} \, q(\bmeta_{\dagger} \mid z, \bx)
            =
            \begin{pmatrix}
                \ad^{-1} + \ad^{-1} (\ln z) \rbr{1-z} \\
                \ad \bx \rbr{1-z}
            \end{pmatrix}.
        \end{equation}

        Under the conditional model in \Cref{eq:model}, the random variable $Z := \bar{z}(Y, \bX)$ has a unit-exponential distribution and is independent of $\bX$. From $Y = \bar{y}(Z, \bX)$, we find
        \begin{align*}
            \ex\sbr{\nabla_{\bmeta} \,m(\bmeta_\dagger \mid Y, \bX)}
            &= \ex\sbr{\nabla_{\bmeta} \,m(\bmeta_\dagger \mid \bar{y}(Z, \bX), \bX)} \\
            &= \ex\sbr{\nabla_{\bmeta} \,q(\bmeta_\dagger \mid Z, \bX)} \\
            & =
            \ex \begin{pmatrix}
                \ad^{-1} + \ad^{-1} (\ln Z) \rbr{1- Z} \\
                \ad \bX \rbr{1- Z}
            \end{pmatrix}
            = \0,
        \end{align*}
        where we used the identities $\ex\sbr{Z \ln Z} = 1 - \gamma$ and $\ex[\ln Z] = - \gamma$, with $\gamma$ the Euler--Mascheroni constant. We have thus verified \textbf{item ii}.

        \emph{Items iii and iv.} The two concern different matrices. \textbf{Item iii} concerns the second moments of the score, which are the entries of the Fisher information matrix $\Fi$ below, so that it is the finiteness of the trace of $\Fi$. \textbf{Item iv} concerns the expected Hessian $\Hes$, the matrix $P \dot\psi_{\theta_0}$ of Theorem~5.41 in \citet{van2000asymptotic}, whose $\psi_\theta$ is our score $\nabla_\bmeta m$. The following result settles both, and shows that the two matrices agree up to sign.

        \begin{lemma}[Fisher information and expected Hessian]
            \label{lem:fisher_information_matrix}
            Assume that $\ex\sbr{\norm{\bX}^2} < \infty$. Then the Fisher information matrix and the expected Hessian matrix associated with the parametric part of the model given by \Cref{eq:model},
            \[
                \Fi = \ex\sbr{\nabla_\bmeta m(\bmeta_\dagger \mid Y, \bX) \, \nabla_\bmeta m(\bmeta_\dagger \mid Y, \bX)^\tr}
                \mtext{and}
                \Hes = \ex\sbr{\nabla_\bmeta^2 m(\bmeta_\dagger \mid Y, \bX)},
            \]
            are well defined, and $\Hes = -\Fi$ with
            \[
            \Fi =
            \begin{pmatrix}
                \dfrac{1}{\ad^2}
                \left(\dfrac{\pi^2}{6} + (1-\gamma)^2 \right)
                & (1 - \gamma)\ex[\bX]^\tr \\[2ex]
                (1 - \gamma)\ex[\bX] & \ad^2 \ex[\bX \bX^\tr]
            \end{pmatrix} .
            \]
            If in addition \Cref{cond:identifiability} is satisfied, then $\Fi$ is positive definite and $\Hes$ is negative definite; in particular, both are nonsingular.
        \end{lemma}

        \noindent \Cref{cond:mgf} supplies the moment condition, so that items~iii and~iv are verified. Finally, the following result is used to verify \textbf{item v}. Recall $\bar{z}(y,\bx)$ in \Cref{eq:barz}.

        \begin{lemma}
            \label{lem:domination_third_derivatives}
            Assume \Cref{cond:mgf} is fulfilled.
            Then there exist constants $0 < r < s < \infty$ and $C, K, \delta \in \Rpos$, depending on $\ad$ only, such that the function $\dddot M$ on $\Rpos \times \Rp$ given by
            \[
                \dddot M(y, \bx)
                =
                C + K \rbr{1 + \norm{\bx}^3 + {\abr{\ln z}}^3} \exp\rbr{\delta \norm{\bx}} \rbr{z^r + z^s}
                \qquad \text{with } z = \bar{z}(y,\bx)
            \]
            uniformly dominates all third-order derivatives of $\bmeta \mapsto m(\bmeta \mid y, \bx)$ in a neighborhood of $\bmeta_\dagger$ and is integrable with respect to the law of $(Y, \bX)$.
        \end{lemma}

        Hence, the assumptions of Theorem~5.41 of \citet{van2000asymptotic} are satisfied. Since $\mle$ is consistent by \Cref{thm:consistency}, and since $-\Hes^{-1} = \Fi^{-1}$ by \Cref{lem:fisher_information_matrix}, we conclude that
        \begin{align*}
            \sqrt n \rbr{\mle - \bmeta_\dagger}
            &= - \Hes^{-1} \frac1{\sqrt n} \sum_{i=1}^{n} \nabla_{\bmeta} m(\bmeta_\dagger \mid Y_i, \bX_i) + o_{\pr}(1) \\
            &= \Fi^{-1} \frac1{\sqrt n} \sum_{i=1}^{n} \nabla_{\bmeta} m(\bmeta_\dagger \mid Y_i, \bX_i) + o_{\pr}(1),
            \qquad n\to\infty.
        \end{align*}
        The summands are iid, centered by \textbf{item ii}, and have covariance matrix $\Fi$ by \Cref{lem:fisher_information_matrix}, so that the central limit theorem and Slutsky's lemma give
        \[
            \sqrt n \rbr{\mle - \bmeta_\dagger} \dto \Normal\rbr{\0, \Fi^{-1} \Fi \Fi^{-1}} = \Normal\rbr{\0, \Fi^{-1}},
            \qquad n \to \infty.
            \qedhere
        \]
    \end{proof}

    \begin{lproof}[Proof of \Cref{lem:fisher_information_matrix}]
        We apply the same change of variables from $y$ to $z = \bar{z}(y, \bx)$ as before. Since the substitution does not involve $\bmeta$, the derivatives of $\bmeta \mapsto m(\bmeta \mid Y, \bX)$ are those of $\bmeta \mapsto q(\bmeta \mid Z, \bX)$, where $Z = \bar{z}(Y, \bX)$ has a unit-exponential distribution and is independent of $\bX$.
        From \Cref{eq:log_likelihood_reparametrization}, we have
        \[
            \nabla_{\bmeta} q(\bmeta_\dagger \mid Z,\bX) =
            \begin{pmatrix}
                \ad^{-1} + \ad^{-1}(1-Z) \, \ln Z \\[2mm]
                \ad \bX(1-Z)
            \end{pmatrix}.
        \]

        \emph{Moments of $Z$.}
        The expectations involving $Z$ below are all derivatives of the gamma function $\Gamma(s) = \int_0^\infty z^{s-1} e^{-z} \, dz$, since $\ex[Z^{s-1} (\ln Z)^k] = \Gamma^{(k)}(s)$ for $s > 0$ and $k \in \cbr{0, 1, 2}$. Differentiating $\Gamma(s+1) = s \, \Gamma(s)$ gives $\Gamma'(s+1) = \Gamma(s) + s \, \Gamma'(s)$ and $\Gamma''(s+1) = 2 \, \Gamma'(s) + s \, \Gamma''(s)$. Starting from $\Gamma(1) = 1$, $\Gamma'(1) = -\gamma$ and $\Gamma''(1) = \gamma^2 + \pi^2/6$, where $\gamma$ is the Euler--Mascheroni constant, these recursions yield $\Gamma(2) = 1$, $\Gamma(3) = 2$ and
        \begin{align*}
            \Gamma'(2) &= 1 - \gamma,
            & \Gamma''(2) &= \frac{\pi^2}{6} - 1 + (1 - \gamma)^2, \\
            \Gamma'(3) &= 3 - 2\gamma,
            & \Gamma''(3) &= \frac{\pi^2}{3} + 2 - 6\gamma + 2\gamma^2.
        \end{align*}

        The Fisher information admits the block-matrix form
        \[
            \Fi = \ex\sbr{\nabla_{\bmeta} q(\bmeta_\dagger \mid Z,\bX) \, \nabla_{\bmeta} q(\bmeta_\dagger \mid Z,\bX)^\tr}
            = \begin{pmatrix}
                \cI_{\alpha\alpha} & \cI_{\alpha\bbeta}^\tr \\
                \cI_{\alpha\bbeta} & \cI_{\bbeta\bbeta}
            \end{pmatrix}
        \]
        along the components of $\bmeta = (\alpha, \bbeta)$, with $\cI_{\alpha\bbeta}$ a column vector in $\Rp$. We calculate each block separately.

        Since $\ex[(1-Z)^2] = \var(Z) = 1$, we have
        \begin{align*}
            \cI_{\bbeta\bbeta} & = \ad^2 \ex\sbr{\bX\bX^\tr(1-Z)^2} \\
            &= \ad^2 \ex[\bX\bX^\tr]\, \ex[(1-Z)^2]\\
            &= \ad^2\ex[\bX\bX^\tr].
        \end{align*}

        Since $\ex[1-Z] = 0$, and since \(\ex[\ln Z] = \Gamma'(1) = -\gamma\), \(\ex[Z\ln Z] = \Gamma'(2) = 1-\gamma\) and \(\ex[Z^2\ln Z] = \Gamma'(3) = 3-2\gamma\), we find
        \begin{align*}
            \cI_{\alpha\bbeta} & = \ex[\bX]\, \ex\sbr{(1-Z)^2\ln Z} \\
            & = (1 - \gamma) \ex[\bX].
        \end{align*}

        From $\ex[(1-Z)\ln Z] = \Gamma'(1) - \Gamma'(2) = -1$, we find
        \begin{align*}
            \cI_{\alpha\alpha} & = \frac1{\ad^2} \ex\sbr{\bigl(1+(1-Z)\ln Z\bigr)^2} \\
            & = \frac1{\ad^2} \bigl( 1 + 2 \ex[(1-Z)\ln Z] + {\ex[(1-Z)^2(\ln Z)^2]}\,\bigr)\\
            & = \frac1{\ad^2} \bigl(-1 + {\ex\sbr{(\ln Z)^2 - 2 \,Z(\ln Z)^2 + Z^2(\ln Z)^2}}\,\bigr)\\
            & = \frac1{\ad^2} \bigl(-1 + {\Gamma''(1) - 2 \, \Gamma''(2) + \Gamma''(3)}\,\bigr)\\
            & = \frac1{\ad^2} \left(\frac{\pi^2}{6} + (1 - \gamma)^2 \right).
        \end{align*}

        Combining the three blocks yields the stated form of $\Fi$.

        \emph{Expected Hessian.}
        Differentiating \Cref{eq:score_general} once more and evaluating at $\bmeta = \bmeta_\dagger$, where $\Delta_{\bbeta} = \0$, so that the argument of $g'$ and $g''$ reduces to $\ln Z$, with $g'(\ln Z) = 1 - Z$ and $g''(\ln Z) = -Z$, we obtain
        \begin{align*}
            \partial^2_{\alpha\alpha} q(\bmeta_\dagger \mid Z, \bX)
            &= -\ad^{-2} \rbr{1 + Z (\ln Z)^2}, \\
            \nabla_{\bbeta} \, \partial_\alpha q(\bmeta_\dagger \mid Z, \bX)
            &= \bX \rbr{1 - Z - Z \ln Z}, \\
            \nabla^2_{\bbeta\bbeta} \, q(\bmeta_\dagger \mid Z, \bX)
            &= -\ad^2 \bX \bX^\tr Z.
        \end{align*}
        Taking expectations, with $\ex[Z] = \Gamma(2)$, $\ex[Z \ln Z] = \Gamma'(2)$ and $\ex[Z (\ln Z)^2] = \Gamma''(2)$, gives
        \begin{align*}
            \ex\sbr{\partial^2_{\alpha\alpha} q(\bmeta_\dagger \mid Z, \bX)}
            &= -\frac{1}{\ad^2} \rbr{\frac{\pi^2}{6} + (1 - \gamma)^2} = -\cI_{\alpha\alpha}, \\
            \ex\sbr{\nabla_{\bbeta} \, \partial_\alpha q(\bmeta_\dagger \mid Z, \bX)}
            &= -(1 - \gamma) \ex[\bX] = -\cI_{\alpha\bbeta}, \\
            \ex\sbr{\nabla^2_{\bbeta\bbeta} \, q(\bmeta_\dagger \mid Z, \bX)}
            &= -\ad^2 \ex[\bX \bX^\tr] = -\cI_{\bbeta\bbeta},
        \end{align*}
        that is, $\Hes = -\Fi$. This is an instance of the information identity. Computing both sides, as done here, only ever takes the expectation of derivatives evaluated pointwise, and so avoids the interchange of differentiation and integration on which the general identity rests.

        That $\Fi$ and $\Hes$ are well defined is the hypothesis $\ex[\norm{\bX}^2] < \infty$. Every entry of the random matrices above is of the form $X_i X_j \, h(Z)$, $X_j \, h(Z)$ or $h(Z)$, with $h$ a polynomial in $z$ and $\ln z$. Since $\ex\abr{h(Z)}$ is finite for every such $h$ and since $Z$ and $\bX$ are independent, such an entry is integrable as soon as $\ex\abr{X_i X_j}$ and $\ex\abr{X_j}$ are finite, which the hypothesis gives through the Cauchy--Schwarz inequality.

        \emph{Positive definiteness.}
        We now show that $\Fi$ is positive definite under \Cref{cond:identifiability}. For a symmetric matrix $\mA$, we write $\mA \succ 0$ and $\mA \succeq 0$ to say that $\mA$ is positive definite and positive semidefinite, respectively. Let
        \[
            M := \ex\sbr{\bX}^\tr \ex\sbr{\bX\bX^\tr}^{-1}\ex\sbr{\bX} \mtext{and} \Sigma := \ex\sbr{\bX\bX^\tr} - \ex\sbr{\bX}\ex\sbr{\bX}^\tr.
        \]
        Since $\Sigma \succeq 0$, we have
        \[
            \rbr{\ex\sbr{\bX\bX^\tr}^{-1}\ex\sbr{\bX}}^{\tr} \Sigma \rbr{\ex\sbr{\bX\bX^\tr}^{-1}\ex\sbr{\bX}} \geq 0
        \]
        and thus
        $
            M - M^2 \geq 0
        $,
        which implies that $0 \leq M \leq 1$. By the Schur complement criterion, a real symmetric block matrix is positive definite as soon as one of its diagonal blocks and the Schur complement of that block are both positive definite. We apply the criterion to the lower right block $\ad^2 \ex\sbr{\bX \bX^\tr}$ of $\Fi$. By the remark after \Cref{cond:mgf}, \Cref{cond:identifiability} implies that $\ex\sbr{\bX \bX^\tr} \succ 0$, so that block is positive definite. Its Schur complement in $\Fi$ is the scalar
        \begin{align*}
            \frac{1}{\ad^2}\rbr{\frac{\pi^2}{6} + \rbr{1-\gamma}^2} - \frac{\rbr{1-\gamma}^2}{\ad^2} \ex\sbr{\bX}^\tr \ex\sbr{\bX \bX^\tr}^{-1} \ex\sbr{\bX}
            & = \frac{1}{\ad^2} \rbr{\frac{\pi^2}{6} + \rbr{1-\gamma}^2 \rbr{1- M}} \\
            & > 0;
        \end{align*}
        the inequality holds because $M \leq 1$. Hence $\Fi \succ 0$, so that $\Hes = -\Fi$ is negative definite, and a definite matrix is nonsingular.
    \end{lproof}

    \begin{lproof}[Proof of \Cref{lem:domination_third_derivatives}]
        Recall $q(\bmeta \mid z, \bx) = m(\bmeta \mid \bar{y}(z, \bx), \bx)$ in \Cref{eq:qbmeta}. Since $m(\bmeta \mid y, \bx) = q(\bmeta \mid \bar{z}(y, \bx), \bx)$ and since $\bar{z}(y, \bx)$ does not depend on the free parameter $\bmeta$ but only on the true parameter $\bmeta_{\dagger}$, we can calculate and bound the third-order partial derivatives of $\bmeta \mapsto m(\bmeta \mid y, \bx)$ via those of $\bmeta \mapsto q(\bmeta \mid z, \bx)$. To show that the bound is integrable, we can then take the expectation with respect to $(Z, \bX)$, where $Z$ has a unit-exponential distribution and is independent of $\bX$. For this reason, we immediately evaluate the third-order partial derivatives in the pair $(Z, \bX)$. Direct calculation yields
        \begin{align*}
            \partial_{\alpha\alpha\alpha}^3 q(\bmeta \mid Z,\bX)
            &= 2\alpha^{-3} -
            \rbr{\Delta_{\bbeta}^\tr \, \bX + \ad^{-1}\ln Z}^3
            \exp\rbr{\alpha\,\Delta_{\bbeta}^\tr \,\bX}
            Z^{\alpha/\ad},
            \\
            \partial_{\alpha\alpha\beta_i}^3 q(\bmeta \mid Z,\bX)
            &=
            -X_i \exp\rbr{\alpha\,\Delta_{\bbeta}^\tr \,\bX}
            Z^{\alpha/\ad}
            \sbr{
                2 \rbr{ \Delta_{\bbeta}^\tr \,\bX + \ad^{-1}\ln Z} + \alpha
                \rbr{\Delta_{\bbeta}^\tr \,\bX +\ad^{-1}\ln Z}^2
            },
            \\
            \partial_{\alpha\beta_i\beta_j}^3 q(\bmeta \mid Z,\bX)
            &= -X_i X_j \exp\rbr{\alpha\,\Delta_{\bbeta}^\tr \,\bX}
            Z^{\alpha/\ad}
            \sbr{
                2\alpha + \alpha^2
                \rbr{\Delta_{\bbeta}^\tr \,\bX + \ad^{-1}\ln Z}
            },
            \\
            \partial_{\beta_i\beta_j\beta_k}^3 q(\bmeta \mid Z,\bX)
            &=
            -\alpha^3X_iX_jX_k
            \exp\rbr{\alpha\,\Delta_{\bbeta}^\tr \,\bX}
            Z^{\alpha/\ad}.
        \end{align*}

        Let $\delta > 0$ be small enough such that $\ex[\exp(\delta \norm{\bX})]$ is finite; such a choice is possible thanks to \Cref{cond:mgf}, as will be explained below.
        Choose $0 < r < \ad < s < \infty$, and define the neighborhood $\Ud$ of $\bmeta_\dagger$ by
        \[
            \Ud = [r,s] \times \cbr{\bbeta \in \Rp : \norm{\bbeta - \bd} \le \delta/(2s)}.
        \]
        For all $\bmeta \in \Ud$, we then have the bound
        \[
            \exp\rbr{\alpha \Delta_{\bbeta}^\tr \bX} Z^{\alpha/\ad}
            \leq \exp\rbr{(\delta/2) \norm{\bX}} \rbr{Z^{r/\ad} + Z^{s/\ad}}
            =: g(Z, \bX).
        \]

        Each of the third-order partial derivatives is thus bounded by an expression in which the term on the left is replaced by the term on the right. Since $\abr{X_i} \leq \norm{\bX}$, we have
        \begin{align*}
            \abr{\partial_{\alpha\alpha\alpha}^3 q(\bmeta \mid Z,\bX)}
            & \leq 2r^{-3} +
            \rbr{\delta \norm{\bX} + \ad^{-1} \abr{\ln Z}}^3 g(Z, \bX)
            ,
            \\
            \abr{\partial_{\alpha\alpha\beta_i}^3 q(\bmeta \mid Z,\bX)}
            & \leq
            \norm{\bX}
            \sbr{
                2 \rbr{ \delta \norm{\bX} + \ad^{-1} \abr{\ln Z}} +
                s \rbr{\delta \norm{\bX} +\ad^{-1}\abr{\ln Z}}^2
            } g(Z, \bX),
            \\
            \abr{\partial_{\alpha\beta_i\beta_j}^3 q(\bmeta \mid Z,\bX)}
            & \leq \norm{\bX}^2
            \sbr{
                2s + s^2
                \rbr{\delta \norm{\bX} + \ad^{-1} \abr{\ln Z}}
            } g(Z, \bX),
            \\
            \abr{\partial_{\beta_i\beta_j\beta_k}^3 q(\bmeta \mid Z,\bX)}
            & \leq
            s^3 \norm{\bX}^3 g(Z, \bX).
        \end{align*}
        For $u,v \ge 0$ and $a,b > 0$ we have $u^a \le 1 + u^{a+1}$ and $u^a v^b \le \max(u,v)^{a+b} \le u^{a+b}+v^{a+b}$.
        Therefore, all factors multiplying $g(Z, \bX)$ in the four inequalities in the display are bounded from above by $K \, (1 + \norm{\bX}^3 + \abr{\ln Z}^3)$ for some positive constant $K$ depending only on $(\ad,\delta,s)$. It follows that all third-order partial derivatives of $\bmeta \mapsto q(\bmeta \mid Z, \bX)$ are bounded in absolute value by
        \begin{equation}
            \label{eq:3bound}
            C + K \, \rbr{1 + \norm{\bX}^3 + \abr{\ln Z}^3} g(Z, \bX),
        \end{equation}
        and this for all $\bmeta \in \Ud$; the constants $C$ and $K$ only depend on $(\ad,\delta,r,s)$. The bound has the same form as the one in the statement of \Cref{lem:domination_third_derivatives} up to a renaming of the constants.

        Since $Z$ and $\bX$ are independent and since $Z$ is exponentially distributed, integrability of the bound in \Cref{eq:3bound} will follow as soon as we can show that we can choose $\delta > 0$ such that $\ex[\exp(\delta \norm{\bX})]$ is finite; then $\ex[\exp((\delta/2) \norm{\bX})]$ and $\ex[\norm{\bX}^a \exp((\delta/2) \norm{\bX})]$ will be finite for all $a > 0$ too. Since the Euclidean norm of a vector is bounded by its sum norm, we have
        \begin{align*}
            \exp\rbr{\delta \norm{\bX}}
            &\le \exp \rbr{\delta |X_1| + \dots + \delta |X_p|} \\
            &\le \max_{\beps \in \cbr{-1,1}^p} \exp \rbr{\delta \eps_1 X_1 + \cdots + \delta \eps_p X_p} \\
            &\le \sum_{\bt \in \cbr{-\delta,\delta}^p} \exp \rbr{ \bt^\tr \bX }.
        \end{align*}
        The number of terms in the sum is $2^p$, and for sufficiently small $\delta > 0$, \Cref{cond:mgf} guarantees that $\ex[\exp(\bt^\tr \bX)]$ is finite for every $\bt \in \cbr{-\delta,\delta}^p$. It then follows that $\ex[\exp(\delta\norm{\bX})]$ is finite too for sufficiently small $\delta$. This concludes the proof.
    \end{lproof}

\end{document}